\documentclass{article}

\usepackage[utf8]{inputenc}
\usepackage[T1]{fontenc}
\usepackage{amsmath,amssymb,amsfonts}
\usepackage{amsthm}
\usepackage{graphicx}
\usepackage[colorlinks,linkcolor=blue,citecolor=blue,urlcolor=blue]{hyperref}
\usepackage[margin=1in]{geometry}
\usepackage{booktabs}

\newtheorem{theorem}{Theorem}[section]
\newtheorem{lemma}[theorem]{Lemma}
\newtheorem{proposition}[theorem]{Proposition}

\newtheorem{definition}[theorem]{Definition}
\newtheorem{claim}[theorem]{Claim}
\theoremstyle{remark}
\newtheorem{remark}[theorem]{Remark}

\newcommand{\Scal}{\mathrm{scal}}
\newcommand{\Ric}{\mathrm{Ric}}
\newcommand{\Rico}{{\mathrm{Ric}_0}}
\newcommand{\CPIC}{C_{PIC}}
\newcommand{\CPICone}{C_{PIC1}}
\newcommand{\CPICtwo}{C_{PIC2}}
\newcommand{\CB}{\mathcal{C}_B}
\newcommand{\dt}{\tfrac{\mathrm{d}}{\mathrm{d}t}}

\title{Pinching cones for positive isotropic curvature in dimensions seven and eight}

\author{Jae Ho Cho}
\date{}

\begin{document}

\maketitle

\begin{abstract}
    We construct and prove the transversal invariance of two families of pinching cones for the
    Hamilton ODE $\dt R=Q(R)$ in dimensions $n=7,8$, thereby extending the pinching estimate that
    Brendle \cite{Brendle2019} established for $n\geq 12$ and Chen \cite{Chen2024} for
    $9\leq n\leq 11$. In dimension $n=8$, two steps in Chen's construction genuinely fail.
    The first, an estimate for the second cone
    family, is repaired by keeping a coupling that is discarded in the dimension-general
    argument. The second is the gluing of the two cone families, for which Chen's estimate
    fails at $n=8$ by a definite margin. We replace it by a new argument which reduces the
    gluing to a single linear inequality on the cone of weakly PIC curvature operators, and
    prove that inequality by an explicit finite combination of frame inequalities. In dimension
    $n=7$, the two separate cone-invariance results are proved, but the gluing step is not
    established here. Accordingly, the resulting full pinching estimate and classification
    statements are restricted to dimension $n=8$. The $n=8$ pinching estimate, together with
    dimension-free arguments
    of Brendle, yields in dimension eight the topological classification
    of compact manifolds with positive isotropic curvature that contain no nontrivial
    incompressible $(n-1)$-dimensional space forms, extending a theorem of Brendle from
    $n\geq 12$. Together with the curvature-improvement and classification results of
    Cho--Li \cite{ChoLi2022} and Brendle--Naff \cite{BrendleNaff}, it yields the classification
    of noncompact $\kappa$-noncollapsed
    ancient solutions to the Ricci flow with uniformly PIC,
    extending a theorem of Cho--Li from $n=4$ or $n\geq 12$.
\end{abstract}

\section{Introduction}
\label{sec:intro}

An algebraic curvature operator $R$ on $\mathbb{R}^n$ is said to have \emph{weakly PIC} (in other
words, nonnegative isotropic curvature) if
\begin{equation}\label{eq:wpic}
    R_{1313}+R_{1414}+R_{2323}+R_{2424}-2R_{1234}\ge 0
\end{equation}
for every orthonormal four-frame $\{e_1,e_2,e_3,e_4\}$. It has \emph{weakly PIC$_1$} if
\begin{equation}\label{eq:wpic12}
    R_{1313}+\lambda^2 R_{1414}+R_{2323}+\lambda^2 R_{2424}-2\lambda R_{1234}\ge 0
\end{equation}
for all such frames and all $\lambda\in[-1,1]$, and \emph{weakly PIC$_2$} if
\begin{equation}\label{eq:wpic2}
    R_{1313}+\lambda^2 R_{1414}+\mu^2 R_{2323}+\lambda^2\mu^2 R_{2424}-2\lambda\mu R_{1234}\ge 0
\end{equation}
for all such frames and all $\lambda,\mu\in[-1,1]$. The corresponding closed convex $O(n)$-invariant cones are
denoted $\CPIC$, $\CPICone$, $\CPICtwo$. We write PIC, PIC$_1$, PIC$_2$ for their interiors (so $(M,g)$
is a PIC manifold when its curvature operator lies in $\operatorname{int}\CPIC$ at every point). These
conditions were introduced by Micallef--Moore \cite{MicallefMoore1988} (see also
Micallef--Wang \cite{MicallefWang1993}) and are central to the
Ricci-flow theory of positive curvature: $\CPIC$, $\CPICone$ and $\CPICtwo$ are all preserved by the
Ricci flow, the first having been proved for $n=4$ by Hamilton \cite{Hamilton1997} and in general
independently by Brendle--Schoen \cite{BrendleSchoen2009} and Nguyen \cite{NguyenLott2010}.

Following B\"ohm--Wilking \cite{BoehmWilking2008} (see also Wilking \cite{Wilking2013}), compact manifolds with positive isotropic
curvature are classified through a family of pinching cones: a continuous family $\{\hat C(\beta)\}$
of closed convex $O(n)$-invariant cones, each transversally invariant under the Hamilton ODE
$\dt R=Q(R)$, entered after a scalar shift of the curvature operator and pinching towards
$\CPICone$. Brendle \cite{Brendle2019} constructed such a family for $n\ge 12$ and deduced the
topological classification of compact PIC manifolds in those dimensions; in dimension four the
corresponding classification was obtained by Hamilton \cite{Hamilton1997}, Chen--Zhu
\cite{ChenZhu2006} and Chen--Tang--Zhu \cite{ChenTangZhu2012}. Chen \cite{Chen2024}
rebalanced Brendle's construction (replacing the fourth defining condition of the first cone
family by a sharper one) and constructed a family of pinching cones for $9\le n\le 11$. The
purpose of the present paper is to carry out the construction in dimension $n=8$ and to draw
the two classification consequences.

The proofs of the two main theorems are given in Section \ref{sec:overview}.

Following Cho--Li \cite[Definition~2.2]{ChoLi2022}, a Riemannian manifold $(M^n,g)$ with
$n\ge5$ is said to have \emph{uniformly PIC} if there exists a constant $\theta>0$ such that
\begin{equation}\label{eq:uniformly-pic}
    R_{1313}+R_{1414}+R_{2323}+R_{2424}-2R_{1234}\ge4\theta\,\Scal(R)>0
\end{equation}
for every orthonormal four-frame. Equivalently, since
$(\mathrm{id}\wedge\mathrm{id})_{ijij}=2$ for $i\ne j$, condition \eqref{eq:uniformly-pic} says
that $\Scal(R)>0$ and $R-\theta'\,\Scal(R)\,\mathrm{id}\wedge\mathrm{id}\in\CPIC$ with $\theta'=\theta/2$;
we use this cone form in Section \ref{sec:overview}.

\begin{theorem}\label{thm:main1}
    Let $n=8$ and let $(M^n,g(t))$, $t\in(-\infty,0]$, be a complete, noncompact,
    $\kappa$-noncollapsed ancient solution of the Ricci flow with uniformly PIC. Then
    $(M^n,g(t))$ has bounded curvature and is isometric either to the shrinking Ricci flow on a
    quotient of the round cylinder $S^{n-1}\times\mathbb R$ by standard isometries, or to the
    Bryant soliton.
\end{theorem}

\begin{theorem}\label{thm:main2}
    Let $n=8$ and let $(M^n,g_0)$ be a compact Riemannian manifold with positive
    isotropic curvature. If $M$ contains no nontrivial incompressible $(n-1)$-dimensional space
    forms, then $M$ is diffeomorphic to a finite connected sum of quotients of $S^n$ and
    $S^{n-1}\times\mathbb R$ by standard isometries.
\end{theorem}

Theorem \ref{thm:main1} extends Theorem $1.2$ of Cho--Li \cite{ChoLi2022}, proved there for
$n=4$ or $n\ge 12$, to dimension $n=8$. Theorem \ref{thm:main2} extends Theorem
$1.4$ of Brendle \cite{Brendle2019} (cf.\ Theorem $1.1$ of Chen \cite{Chen2024} for
$9\le n\le 11$) to $n=8$. Both deductions are, after the pinching estimate is in hand,
insensitive to the dimension: Theorem \ref{thm:main2} is obtained from the pinching estimate
by the surgery and finite-extinction machinery of \cite{Brendle2019} (valid in dimension
$n\ge 5$ once the pinching estimate is available, see \cite[\S1.5]{Chen2024}), and Theorem \ref{thm:main1} follows from the
curvature-improvement arguments of \cite{ChoLi2022} together with the classification of ancient
$\kappa$-solutions of Brendle--Naff \cite{BrendleNaff} (see also
\cite{BamlerCabezasRivasWilking2019,LiNi2020}). For related rigidity results for gradient shrinking solitons with positive
isotropic curvature, see Naff \cite{Naff2019} and Chen \cite{Chen2025}; in dimension four, see
Ni--Wallach \cite{NiWallach2007}, Li--Ni--Wang \cite{LiNiWang2018} and Li--Wang
\cite{LiWang2019}; and for the round cylinders themselves, see the rigidity theorems of
Li--Wang \cite{LiWang2024} and Li--Zhang \cite{LiZhang2023}. The underlying structure theory
of noncollapsed shrinkers is developed in \cite{LiLiWang2021,DongLi2026} and further in
\cite{LiZhang2026}, and the K\"ahler analogue of the ancient-solution classification is treated
in \cite{Li2024} (see also \cite{LiZhang2023b}). Accordingly the only dimension-restricted
input to either theorem is the pinching estimate, which is the content of this paper.

\paragraph{Outline of the pinching-cone construction.}
The construction consists of two families of cones, $C(b)$ for
$0<b\le b_{max,2}$ (Subsection \ref{subsec:first}) and $\tilde C(b)$ for
$0<b\le\tilde b_{max,2}$ (Subsection \ref{subsec:second}), which are glued along a
common endpoint (Subsection \ref{sec:gluing}). Concatenating them gives a continuous family
$\hat C(\beta)$ for $0<\beta<B_n$ (see Proposition \ref{prop:pinching} for the value of $B_n$),
each transversally
invariant under $\dt R=Q(R)$, with the entry property stated below, and
pinching towards $\CPICone$ as $\beta\nearrow B_n$. This is recorded, together with the pinching
estimate it implies, in Proposition \ref{prop:pinching}. Every dimension-dependent input of the
construction is a scalar inequality, and at $n=8$ these inequalities fall into three kinds.

\begin{enumerate}
    \item[(a)] The scalar inequalities behind the first family (Lemmas \ref{lem:monotonicity-8},
    \ref{lem:gapcond-8} and \ref{lem:gapcond2-8}) hold at $n=8$ with small but definite margins.
    At $n=7$, a sharper endpoint comparison verifies the first-family inequalities
    (Subsection \ref{subsec:proof1-7} and Remark \ref{rem:n7reduction}).
    \item[(b)] The estimate controlling the second family (Lemma $4.3$ of \cite{Chen2024}) fails
    at $n=8$: in the notation of Subsection \ref{subsec:proof2-common}, one has
    $1+(n-2)(1-\zeta)-2K\zeta^2<0$ at $b=\tilde b_{max,2}$. We observe that the estimate in which it
    is used only requires the weaker inequality
    $1+(n-2)(1-\zeta)-2K\zeta^2\,\frac{b}{\tilde a}>0$, in which the factor $\frac{b}{\tilde
    a}=(1+\frac{n-2}{2}b)^{-1}$ is retained rather than discarded. The resulting quantity
    $F(b)$ is positive at $n=8$ (Remark \ref{rem:sharpbracket}, and
    \eqref{eq:F8end}).
    \item[(c)] The gluing of the two families is controlled, in \cite{Chen2024}, by a scalar
    estimate (Lemma $4.2$, equation $(4.1)$ there) which at $n=8$ evaluates numerically to
    $0.2688\ldots\ge0.3181\ldots$ and is therefore false. This is the essential obstruction.
    Section \ref{sec:proof3} replaces the estimate by a different argument: after the reduction
    of Subsection \ref{subsec:proof3-reduction} the gluing becomes a one-parameter family
    (indexed by $\lambda\in[0,1]$) of linear inequalities for weakly PIC operators. We show that the
    whole family follows, through a Bernstein representation of the underlying quadratic
    polynomial in $\lambda$, from the frame inequality at $\lambda=1$, the tilted-frame
    inequalities, and a single linear inequality, which we prove by exhibiting it as an explicit
    nonnegative combination of eight averaged frame inequalities (Subsection
    \ref{subsec:proof3-8}).
\end{enumerate}

This paper is organized as follows. Section \ref{sec:preliminaries} recalls the B\"ohm--Wilking transformations and collects the
standard consequences of weakly PIC that are used throughout. Section \ref{sec:overview}
states the three transversal-invariance and gluing theorems which constitute the
core technical result and derives the pinching estimate and the two main theorems from them.
Sections \ref{sec:proof1}--\ref{sec:proof3} verify the three theorems at $n=8$.
The separate invariance results in dimension seven are retained in Subsections
\ref{subsec:proof1-7} and \ref{subsec:proof2-7}; the limitation at the gluing step is
explained in the \hyperref[note:dimension-seven]{note on dimension seven} in Section \ref{sec:proof3}.

\paragraph{AI disclosure.}
AI tools were used to assist with the exact computations behind the polynomial inequalities and
the linear certificates in this paper, and to help with the copyediting of the manuscript. The
polynomial identities and the linear certificate are verified explicitly in the proofs.

\section{Preliminaries}
\label{sec:preliminaries}

We collect the algebraic setup and the standard consequences of weakly PIC
that are used throughout. Everything in this section is independent of the dimension apart from the
requirement $n\ge6$ in Lemma \ref{lem:wpic}. In particular it applies to $n=7,8$.

\subsection{The B\"ohm--Wilking transformation}
\label{subsec:bw}

Let $\CB(\mathbb R^n)\subset S^2(\mathfrak{so}(n))$ denote the space of algebraic curvature operators, i.e.\
symmetric bilinear forms on $\Lambda^2\mathbb R^n$ satisfying the first Bianchi identity. The space
$\CB(\mathbb R^n)$ carries the orthogonal $O(n)$-invariant decomposition
\begin{equation}\label{eq:cb-decomp}
    \CB(\mathbb R^n)=\langle\mathrm{id}\wedge\mathrm{id}\rangle\oplus\langle\Rico\rangle\oplus\langle\mathcal W\rangle
\end{equation}
into pairwise inequivalent irreducible $O(n)$-invariant subspaces. Here
$\langle\mathrm{id}\wedge\mathrm{id}\rangle$ denotes multiples of the identity,
$\langle\Rico\rangle$ the curvature operators of traceless Ricci type, and
$\langle\mathcal W\rangle$ the space of Weyl curvature operators. Given $a,b$ with $0\le b\le a$, the
B\"ohm--Wilking map $l_{a,b}:\CB(\mathbb R^n)\to\CB(\mathbb R^n)$ is the $O(n)$-equivariant linear map acting
by scalars on these three summands. Equivalently,
\begin{equation}\label{eq:lab-form}
    l_{a,b}(R)=R+b\,\Ric(R)\wedge\mathrm{id}+\frac{2(a-b)}{n}\Scal(R)\,I,\qquad
    I=\tfrac12\,\mathrm{id}\wedge\mathrm{id},
\end{equation}
with the trace identities
\begin{equation}\label{eq:lab-traces}
    \Rico(l_{a,b}(R))=(1+(n-2)b)\Rico(R),\qquad
    \Scal(l_{a,b}(R))=(1+2(n-1)a)\Scal(R).
\end{equation}
The Hamilton ODE on $\CB(\mathbb R^n)$ is $\dt R=Q(R)$, where
$Q(R)_{ijkl}=R_{ijpq}R_{klpq}+2R_{ipkq}R_{jplq}-2R_{iplq}R_{jpkq}$ is the quadratic form on
$S^2(\mathfrak{so}(n))$ induced by the Lie bracket on $\mathfrak{so}(n)$ (the algebraic curvature operator
of the reaction term of the Ricci flow, see \cite{BoehmWilking2008,Chen2024}). Pulling this ODE back by
$l_{a,b}$, a curvature operator $R=l_{a,b}(S)$ solves $\dt R=Q(R)$ if and only if $S$ solves
\begin{equation}\label{eq:pullback-ode}
    \dt S=Q(S)+D_{a,b}(S),
\end{equation}
where
\begin{equation}\label{eq:Dab-form}
    \begin{split}
        D_{a,b}(S)={}&\bigl(2b+(n-2)b^2-2a\bigr)\Rico(S)\wedge\Rico(S)
        +2a\,\Ric(S)\wedge\Ric(S)+2b^2\Rico(S)^2\wedge\mathrm{id}\\
        &+\frac{nb^2(1-2b)-2(a-b)(1-2b+nb^2)}{n\bigl(1+2(n-1)a\bigr)}\,
        |\Rico(S)|^2\,\mathrm{id}\wedge\mathrm{id}.
    \end{split}
\end{equation}
Taking the trace and using Hamilton's identity $\Scal(Q(S))=2|\Ric(S)|^2$ gives the scalar evolution
\begin{equation}\label{eq:scalar-evol}
\begin{split}
    \dt\Scal(S)&=P\,|\Ric(S)|^2+Q\,\Scal(S)^2,\\
    P&=\frac{2(1+(n-2)b)^2}{1+2(n-1)a},\qquad
    Q=\frac{2\bigl((1+2(n-1)a)^2-(1+(n-2)b)^2\bigr)}{n\bigl(1+2(n-1)a\bigr)},
\end{split}
\end{equation}
with $P+nQ=2+4(n-1)a$. The evolution of the Ricci tensor is recorded for later use.

\begin{proposition}[cf.\ Section 2 of \cite{Chen2024}]\label{prop:bw}
    If $S$ evolves by \eqref{eq:pullback-ode}, then
    \begin{equation}\label{eq:ricci-evol}
        \begin{split}
            \dt\Ric(S)={}&2\bigl(S*\Ric(S)\bigr)-4b\,\Ric(S)^2
            +\frac4n\bigl(2b+(n-2)a\bigr)\Scal(S)\Ric(S)\\
            &+2\,\frac{n^2b^2-2(n-1)(a-b)(1-2b)}{n\bigl(1+2(n-1)a\bigr)}\,|\Rico(S)|^2\mathrm{id}
            +\frac4{n^2}(a-b)\Scal(S)^2\mathrm{id},
        \end{split}
    \end{equation}
    where $(S*\Ric)_{ik}=\sum_{p,q}S_{ipkq}\Ric_{pq}$, and $\dt\Scal(S)$ is given by
    \eqref{eq:scalar-evol}. In particular, $P+nQ=2+4(n-1)a$.
\end{proposition}

\begin{proof}
    Equation \eqref{eq:ricci-evol} follows from the definition of $D_{a,b}$ together with the
    Hamilton identities $\Ric(Q(S))=2S*\Ric(S)$ and $\Scal(Q(S))=2|\Ric(S)|^2$. Taking the trace of
    \eqref{eq:ricci-evol} and using $|\Rico|^2=|\Ric|^2-\frac1n\Scal^2$ gives \eqref{eq:scalar-evol},
    and the displayed formula for $P$ yields $P+nQ=2+4(n-1)a$. See \cite[\S2]{Chen2024}.
\end{proof}

\subsection{Elementary consequences of weakly PIC}
\label{subsec:wpic}

We use the following frame inequalities repeatedly. If $\{v_1,v_2,v_3,v_4\}$ is an orthonormal
four-frame we write, specializing \eqref{eq:wpic},
\begin{equation}\label{eq:wpic-frame}
    W(v_1,v_2,v_3,v_4):=S_{1313}+S_{1414}+S_{2323}+S_{2424}-2S_{1234}\ge0,
\end{equation}
where here and below $e_1,e_2,e_3,e_4$ are any relabelling of $v_1,\dots,v_4$. Replacing $v_4$ by
$-v_4$ leaves the first four terms unchanged and reverses the sign of the last. Averaging the two
inequalities kills the last term. This sign averaging is used without further comment.

\begin{lemma}\label{lem:wpic}
    Let $n\ge6$ and $S\in\CPIC\cap\CB(\mathbb R^n)$. Then, for every orthonormal frame
    $\{e_1,\dots,e_n\}$:
    \begin{enumerate}
        \item $\Scal(S)-2\Ric_{kk}=\sum_{i\ne k}\sum_{j\ne k}S_{ijij}\ge0$ for every $k$;
        \item $\sum_{p\ge3}(S_{1p1p}+S_{2p2p})\ge0$; equivalently
        $\Ric_{11}+\Ric_{22}\ge2S_{1212}$;
        \item $\Ric_{11}+\Ric_{22}+\Ric_{33}+\Ric_{44}\ge0$;
        \item $(n^2-7n+14)S_{1212}+(n-5)(\Ric_{11}+\Ric_{22})+\Scal(S)\ge0$; for $n=8$ this reads
        $22S_{1212}+3(\Ric_{11}+\Ric_{22})+\Scal(S)\ge0$, and together with (2) it gives
        $\Ric_{11}+\Ric_{22}\ge-\frac1{14}\Scal(S)$;
        \item for every $\lambda\in[0,1]$ and every $p\ge5$,
        $S_{1313}+\lambda^2S_{1414}+S_{2323}+\lambda^2S_{2424}-2\lambda S_{1234}
        +(1-\lambda^2)(S_{1p1p}+S_{2p2p})\ge0$;
        \item $\Scal(S)\ge0$, with equality if and only if $S=0$.
    \end{enumerate}
\end{lemma}

\begin{proof}
    Write $K_{ij}=S_{ijij}$. Sign averaging \eqref{eq:wpic-frame} gives
    \begin{equation}\label{eq:wpic-sign-average}
        W^0(i,j,p,q):=K_{ip}+K_{iq}+K_{jp}+K_{jq}\ge0
    \end{equation}
    for four distinct indices. For an index set $J$ of size $m\ge4$, sum the three pairings
    of each four-element subset of $J$. Each $K_{ij}$, $i<j$ in $J$, occurs
    $2\binom{m-2}{2}$ times, so
    \begin{equation}\label{eq:wpic-subspace-trace}
        \sum_{\substack{i<j\\i,j\in J}}K_{ij}\ge0.
    \end{equation}
    Taking $J=\{1,\dots,n\}\setminus\{k\}$ proves (1); taking $J=\{1,\dots,n\}$ proves
    $\Scal(S)\ge0$.

    Put $X=\sum_{p\ge3}(K_{1p}+K_{2p})$ and $U=\Ric_{11}+\Ric_{22}=2K_{12}+X$.
    Summing $W^0(1,2,p,q)$ over $3\le p<q\le n$ gives $(n-3)X\ge0$, proving (2).

    For (3), put
    \begin{equation*}
        A_0=\{1,2,3,4\},\qquad B_0=\{5,\dots,n\},\qquad m=|B_0|=n-4\ge2.
    \end{equation*}
    By \eqref{eq:wpic-subspace-trace}, $H:=\sum_{i<j,\ i,j\in A_0}K_{ij}\ge0$.
    Summing $W^0(i,j,p,q)$ over unordered pairs $i<j$ in $A_0$ and $p<q$ in $B_0$ gives
    $3(m-1)D\ge0$, where $D:=\sum_{i\in A_0,\ p\in B_0}K_{ip}$.
    Hence $\sum_{i=1}^4\Ric_{ii}=2H+D\ge0$.

    For (4), sum $W^0(1,p,2,q)$ over ordered pairs $p\ne q$ in $\{3,\dots,n\}$:
    \begin{equation*}
    \begin{split}
        0&\le (n-2)(n-3)K_{12}+(n-3)X+2\sum_{3\le p<q\le n}K_{pq}\\
         &=(n^2-7n+14)K_{12}+(n-5)U+\Scal(S).
    \end{split}
    \end{equation*}
    At $n=8$, adding eleven times $U-2K_{12}\ge0$ gives $14U+\Scal(S)\ge0$.

    Part (5) is the average of \eqref{eq:wpic-frame} over the two orthonormal four-frames
    $\{e_1,e_2,e_3,\lambda e_4\pm\sqrt{1-\lambda^2}\,e_p\}$.
    Finally, if $\Scal(S)=0$, (1) gives $\Ric(v,v)\le0$ for every unit vector $v$.
    Since $\operatorname{tr}\Ric=0$, this implies $\Ric=0$. Part (2) then gives $K_{ij}\le0$
    for every orthonormal pair. In any orthonormal basis their sum is $\Scal(S)/2=0$, so
    all of them vanish. Every orthonormal pair extends to such a basis; hence all sectional
    curvatures vanish, and polarization of an algebraic curvature tensor gives $S=0$.
    This proves (6).
\end{proof}

\section{Overview of the proof}
\label{sec:overview}

    \subsection{The first family of pinching cones}
    \label{subsec:first}

    We adapt the first family of pinching cones constructed by Chen \cite{Chen2024} (Definition 3.1
    there; cf.\ Definition 3.1 of \cite{Brendle2019}, whose fourth defining condition is replaced by a
    sharper one). For $n=7,8$, define
    \begin{equation*}
        (\varepsilon,b_{max,2})=
        \begin{cases}
            (10^{-3},\frac1{18}),&n=8,\\
            (10^{-4},\frac{131}{2000}),&n=7.
        \end{cases}
    \end{equation*}
    For $0<b\le b_{max,2}$ we associate to $b$ the data $a,\gamma,\rho,\omega,A,P,Q$ by
    \begin{equation}\label{data}
        \begin{split}
            a&=a(b)=\frac{(2+(n-2)b)^2}{2(2+(n-3)b)}b,\\
            \gamma&=\gamma(b)=\frac{b}{2+(n-3)b},\\
            \rho&=\rho(b)=b-\frac{2(n-1)\gamma(1-2b)}{n^2}-\frac{2(n-1)(1+\gamma)(n^2b^2-2(n-1)(a-b)(1-2b))}{n^2(1+2(n-1)a)},\\
            \omega&=\omega(b)=(1-\varepsilon)\sqrt{\frac{27(2+(n-2)b)}{8}\frac{b(1+(n-2)b)^2}{n^2\rho^3(2+(n-3)b)^2}},\\
            A&=A(b)=\frac{2+8b}{(n-1)(n-4)}+\frac{4}{n}(2b+(n-2)a),\\
            P&=P(b)=\frac{2(1+(n-2)b)^2}{1+2(n-1)a},\\
            Q&=Q(b)=\frac{2((1+2(n-1)a)^2-(1+(n-2)b)^2)}{n(1+2(n-1)a)}.
        \end{split}
    \end{equation}
    As computed in Section 2 of \cite{Chen2024}, if an algebraic curvature operator $S$ evolves by the
    ODE $\dt S=Q(S)+D_{a,b}(S)$, then
    \begin{equation*}
        \dt\Scal(S)=P|\Ric(S)|^2+Q\Scal(S)^2 .
    \end{equation*}

    \begin{definition}[Definition 3.1 of \cite{Chen2024}]\label{def:pinchingCones1}
        For $0<b\le b_{max,2}$, let $\mathcal{E}(b)$ denote the set of all algebraic curvature tensors $S$
        for which there exists a tensor $T\in S^2(\mathfrak{so}(n))$ satisfying:
        \begin{enumerate}
            \item $T\geq 0$;
            \item $S-T\in\CPIC$ in the relaxed sense of Remark \ref{rem:relaxed} below;
            \item $\Ric(S)_{11}+\Ric(S)_{22}+\frac{2\gamma}{n}\Scal(S)\geq 0$ for every orthonormal pair
            $\{e_1,e_2\}$;
            \item for every orthonormal frame $\{e_1,\dots,e_n\}$,
            \begin{equation}\label{eq:gapcond}
                \Ric(S)_{22}-\Ric(S)_{11}\leq\omega^{\frac12}\Scal(S)^{\frac12}
                \Bigl(\sum_{p=3}^n(T_{1p1p}+T_{2p2p})\Bigr)^{\frac12}.
            \end{equation}
        \end{enumerate}
        We then define $C(b)=l_{a,b}(\mathcal{E}(b))$.
    \end{definition}

    \begin{remark}\label{rem:relaxed}
        As in \cite{Chen2024}, we say that $T\in S^2(\mathfrak{so}(n))$ lies in $\CPIC$ if
        $T(\varphi,\bar\varphi)\geq 0$ for every $\varphi=(e_1+ie_2)\wedge(e_3+ie_4)$ with
        $\{e_1,e_2,e_3,e_4\}$ orthonormal, i.e.\ if
        $T_{1313}+T_{1414}+T_{2323}+T_{2424}+2T_{1342}+2T_{1423}\geq 0$ for all orthonormal four-frames;
        for tensors satisfying the first Bianchi identity this coincides with the weakly PIC condition. We
        write $T\geq 0$ if $T$ is nonnegative definite as a symmetric bilinear form on
        $\Lambda^2\mathbb{R}^n$.
    \end{remark}

    The proof is given in Section \ref{sec:proof1}.

    \begin{theorem}\label{thm:trans1}
        Let $n=7$ or $n=8$. Then, for each $0<b\le b_{max,2}$, the cone $C(b)$ is transversally
        invariant under the Hamilton ODE $\dt R=Q(R)$.
    \end{theorem}

    \subsection{The second family of pinching cones}
    \label{subsec:second}

    Define
    \begin{equation*}
        \tilde b_{max,2}=\begin{cases}
            \frac1{40},&n=8,\\
            \frac3{125},&n=7.
        \end{cases}
    \end{equation*}

    \begin{definition}[Definition 4.1 of \cite{Chen2024}, with the omitted $\lambda^2S_{2424}$ term restored; cf.\ Definition 4.1 of \cite{Brendle2019}]\label{def:pinchingCones2}
        Assume $0<b\le\tilde b_{max,2}$ and let $a=b+\frac{n-2}{2}b^2$. We denote by
        $\tilde{\mathcal{E}}(b)$ the set of all $S\in\CB(\mathbb{R}^n)$ such that (i)
        $l_{a,b}(S)\in C(b_{max,2})$ and (ii)
        \begin{equation}\label{eq:Z}
            Z:=S_{1313}+\lambda^2S_{1414}+S_{2323}+\lambda^2S_{2424}-2\lambda S_{1234}
            +\sqrt{2a}(1-\lambda^2)\bigl(\Ric(S)_{11}+\Ric(S)_{22}\bigr)\geq 0
        \end{equation}
        for every orthonormal four-frame $\{e_1,e_2,e_3,e_4\}$ and every $\lambda\in[0,1]$. We then
        define $\tilde C(b)=l_{a,b}(\tilde{\mathcal{E}}(b))$.
    \end{definition}

    It is clear that $\tilde C(b)$ is convex for each $b$.

    The proof is given in Section \ref{sec:proof2}.

    \begin{theorem}\label{thm:trans2}
        Let $n=7$ or $n=8$. Then, for each $0<b\le\tilde b_{max,2}$, the cone $\tilde C(b)$ is
        transversally invariant under the Hamilton ODE $\dt R=Q(R)$.
    \end{theorem}

    \subsection{Gluing the two families of pinching cones}
    \label{sec:gluing}

    The proof is given in Section \ref{sec:proof3}.

    \begin{theorem}\label{thm:gluing}
        Let $n=8$. Then for $b_{max,2}$ and $\tilde{b}_{max,2}$ as in Theorems \ref{thm:trans1} and \ref{thm:trans2}, we have
        $C(b_{max,2})=\tilde C(\tilde b_{max,2})$.
    \end{theorem}

    Concatenating the two families along this common endpoint gives a single pinching family, and
    with it the pinching estimate. The two main theorems are then consequences of dimension-free
    machinery applied to this estimate. We record both steps here, since they use none of the
    dimension-restricted input of the cone construction.

    \begin{proposition}[Pinching estimate]\label{prop:pinching}
        Let $n=8$ and let $C(b)$, $\tilde C(b)$, $b_{max,2}$ and $\tilde b_{max,2}$ be as in
        Definitions \ref{def:pinchingCones1} and \ref{def:pinchingCones2}. Put
        $B_n:=b_{max,2}+\tilde b_{max,2}$ and define
        \begin{equation}\label{eq:hatC}
            \hat C(\beta):=
            \begin{cases}
                C(\beta),& 0<\beta\le b_{max,2},\\
                \tilde C(B_n-\beta),& b_{max,2}<\beta<B_n.
            \end{cases}
        \end{equation}
        Then $\{\hat C(\beta)\}_{0<\beta<B_n}$ is a continuous family of closed convex
        $O(n)$-invariant cones, each transversally invariant under $\dt R=Q(R)$, with the following two
        properties. All Ricci inequalities in the sets below are required for every orthonormal pair.
        \begin{enumerate}
            \item \emph{(Entry.)} For every $\theta>0$ and $N\ge0$ there exists $\beta_0\in(0,b_{max,2}]$ such that
            \begin{equation*}
            \begin{split}
                \{R:\Scal(R)\ge0,\ R-\theta\Scal(R)\,\mathrm{id}\wedge\mathrm{id}\in\CPIC\}&\cap
                \{R:\Ric(R)_{11}+\Ric(R)_{22}-\theta\Scal(R)+N\ge0\}\\
                &\subset\{R:R+N\,\mathrm{id}\wedge\mathrm{id}\in\hat C(\beta_0)\}.
            \end{split}
            \end{equation*}
            \item \emph{(Pinching.)} If $\beta_j\nearrow B_n$, there are $\varepsilon_j\searrow0$ with
            \begin{equation*}
                \hat C(\beta_j)\subset
                \{R:R+\varepsilon_j\Scal(R)\,\mathrm{id}\wedge\mathrm{id}\in\CPICone\}.
            \end{equation*}
        \end{enumerate}
        Consequently, for every compact $K\subset\operatorname{int}\CPIC$ and every $T>0$ there exist
        $\theta>0$, $N>0$, an increasing concave function $f>0$ with $\lim_{s\to\infty}f(s)/s=0$, and a
        continuous family $\{F_t\}_{t\in[0,T]}$ of closed convex $O(n)$-invariant sets, invariant under
        $\dt R=Q(R)$, with $K\subset F_0$ and
        \begin{equation}\label{eq:pinchest}
        \begin{split}
            F_t\subset{}&\{R:R-\theta\Scal(R)\,\mathrm{id}\wedge\mathrm{id}\in\CPIC\}\cap
            \{R:\Ric(R)_{11}+\Ric(R)_{22}-\theta\Scal(R)+N\ge0\}\\
            &\cap\{R:R+f(\Scal(R))\,\mathrm{id}\wedge\mathrm{id}\in\CPICtwo\}
        \end{split}
        \end{equation}
        for all $t\in[0,T]$.
    \end{proposition}

    As in the standard pinching-cone construction of \cite[\S5]{Brendle2019} (see also
    \cite[\S1.3]{Chen2024}), the defining inequalities give a continuous family of closed,
    convex, $O(n)$-invariant cones; the dimension-dependent choices of $\varepsilon$ and the
    endpoints do not affect this structural argument. Every nonzero member has positive scalar
    curvature, and $I$ lies in the interior of each cone, by the defining inequalities and the
    trace identities \eqref{eq:lab-traces}. In particular, every compact subfamily,
    after affine reparameterization, satisfies condition $(*)$ of \cite[Thm.~3.3]{ChoLi2022}.
    Transversal invariance follows from
    Theorems \ref{thm:trans1} and \ref{thm:trans2}, and the two branches match at
    $\beta=b_{max,2}$ by Theorem \ref{thm:gluing}.

    For entry, put $H=\mathrm{id}\wedge\mathrm{id}$ and $r=\Scal(R)$. The case $r+N=0$
    has $R=0$ by Lemma \ref{lem:wpic}(6). Otherwise set
    \begin{equation*}
        R_*:=\frac{R+NH}{r+N},\qquad t:=\frac{r}{r+N}\in[0,1].
    \end{equation*}
    The hypotheses give, for every orthonormal pair,
    \begin{equation*}
    \begin{gathered}
        R_*-(\theta t+1-t)H\in\CPIC,\qquad
        \Ric(R_*)_{11}+\Ric(R_*)_{22}\ge\theta t+(4n-5)(1-t),\\
        \Scal(R_*)=t+2n(n-1)(1-t)\in[1,2n(n-1)].
    \end{gathered}
    \end{equation*}
    Lemma \ref{lem:wpic}(6) implies that bounded scalar-curvature slices of $\CPIC$ are compact:
    an unbounded sequence, divided by its norm, would limit to a nonzero weakly PIC tensor
    of scalar curvature zero. Thus all these $R_*$ lie in a fixed compact subset of
    $\operatorname{int}\CPIC$, with a uniform positive lower bound for the sum of any two
    Ricci eigenvalues. As $b\to0$, $l_{a(b),b}^{-1}\to\mathrm{id}$ and $\omega(b)\to\infty$;
    the latter follows from
    \begin{equation*}
        \frac{\rho(b)}b\longrightarrow\frac{n^2-n+1}{n^2}=\frac{57}{64}.
    \end{equation*}
    Consequently, for all sufficiently small $b>0$, the tensors $S_b=l_{a(b),b}^{-1}(R_*)$
    have uniformly positive scalar curvature and two-eigenvalue sums, and
    $S_b-\eta H\in\CPIC$ for a fixed $\eta>0$. Choose $T=\eta H$ in Definition
    \ref{def:pinchingCones1}. Conditions (1)--(3) hold, while
    \begin{equation*}
        \sum_{p=3}^n(T_{1p1p}+T_{2p2p})=4(n-2)\eta>0.
    \end{equation*}
    The Ricci differences of $S_b$ are uniformly bounded, so $\omega(b)\to\infty$ gives
    condition (4) for all sufficiently small $b$. Fixing one such $b=\beta_0$ and using
    homogeneity proves entry, also when $N=0$.
    The scalar-curvature assumption is automatic when $0<\theta<1/(2n(n-1))$, since
    $\Scal(R-\theta rH)=(1-2n(n-1)\theta)r\ge0$. This is the small-parameter setting of
    the entry argument in the proof of \cite[Theorem~5.1]{Brendle2019}.

    For the pinching property, normalize $R\in\tilde C(\tilde b)\setminus\{0\}$ by
    $\Scal(R)=1$. Since $\tilde C(\tilde b)\subset C(b_{max,2})$, these tensors lie in a fixed
    compact set. Put $S=l_{\tilde a,\tilde b}^{-1}(R)$. As $\tilde b\to0$, both $S$ and its
    Ricci tensor are uniformly bounded, and $S-R\to0$ uniformly. Thus \eqref{eq:Z} implies
    that the left-hand side of \eqref{eq:wpic12}, evaluated on $R$, is at least
    $-\delta(\tilde b)$ for every frame and $\lambda\in[0,1]$, where $\delta(\tilde b)\ge0$
    and $\delta(\tilde b)\to0$. Its value on $H$ is $4(1+\lambda^2)\ge4$, so
    $R+\frac14\delta(\tilde b)H\in\CPICone$; negative $\lambda$ follow by reflecting $e_4$.
    Homogeneity and a decreasing upper envelope give the stated sequence $\varepsilon_j$.
    The family
    $\{F_t\}$ is then obtained from $\{\hat C(\beta)\}$ by the continuous-fit procedure of Steps $1$
    and $3$ of \cite[\S1]{Brendle2019} (see \cite[\S1.5]{Chen2024}), which uses only the two
    properties just stated.

    \paragraph{Derivation of Theorem \ref{thm:main1} (ancient solutions).}
    Let $(M^n,g(t))$, $t\in(-\infty,0]$, be as in Theorem \ref{thm:main1}. Uniformly PIC gives a
    constant $\theta>0$ with $R(g(t))-\theta\Scal(g(t))\,\mathrm{id}\wedge\mathrm{id}\in\CPIC$ for all $t$.
    We use the following imported results. (I1) A complete noncompact $\kappa$-noncollapsed ancient
    solution with uniformly PIC has $2$-positive Ricci curvature,
    $\Ric_{11}+\Ric_{22}\ge\delta\Scal$ for some $\delta>0$ \cite[Prop.~3.1]{ChoLi2022}.
    (I2) If the curvature operator lies in $\hat C(\beta_0)$ for some $\beta_0>0$, the
    cone-continuation theorem places it in $\hat C(\beta)$ for every $\beta<B_n$
    \cite[Thm.~3.3]{ChoLi2022}. (I3) A nonflat ancient solution with weakly PIC$_1$ has weakly
    PIC$_2$ \cite[Prop.~6.2]{LiNi2020}; compare \cite[Lemma~4.2]{BamlerCabezasRivasWilking2019},
    which assumes bounded curvature. (I4) Under completeness, noncompactness,
    $\kappa$-noncollapsing, uniformly PIC and weakly PIC$_2$, bounded curvature follows from
    \cite[Prop.~5.6]{ChoLi2022}, and the solution is a quotient of the shrinking round cylinder
    $S^{n-1}\times\mathbb R$ by standard isometries or the Bryant soliton
    \cite[Thm.~5.7]{ChoLi2022}; see also \cite{BrendleNaff}. By (I1) and the entry
    property with $N=0$, using $\Scal(g(t))>0$ and shrinking $\theta\le\delta$, $R(g(t))\in\hat C(\beta_0)$ for
    some $\beta_0>0$ and all $t$. Then (I2) and the pinching property give
    $R(g(t))\in\CPICone$ for all $t$. Thus (I3) gives weakly PIC$_2$, and (I4) supplies bounded
    curvature and the stated classification.

    \paragraph{Derivation of Theorem \ref{thm:main2} (compact classification).}
    Let $(M^n,g_0)$ be as in Theorem \ref{thm:main2} and scale $g_0$ so its curvature operator lies
    in a fixed compact $K\subset\operatorname{int}\CPIC$. Proposition \ref{prop:pinching} supplies the pinching
    estimate \eqref{eq:pinchest} for arbitrarily large $T$. As observed by Chen \cite[\S1.5]{Chen2024},
    in Brendle's program the ancient-solution analysis is valid for all $n\ge5$, and the surgery
    process uses the assumption $n\ge12$ only through the pinching estimate; Proposition
    \ref{prop:pinching} supplies the required estimate at $n=8$. The Ricci-flow-with-surgery and
    finite-extinction machinery of \cite{Brendle2019} therefore applies and produces a finite surgery
    program after which the flow becomes extinct. The topology thereby obtained is read off from the
    analysis of such finite-surgery programs in \cite{Brendle2019}, giving the stated diffeomorphism
    type as a finite connected sum of
    quotients of $S^n$ and $S^{n-1}\times\mathbb R$ by standard isometries. The hypothesis on
    incompressible $(n-1)$-dimensional space forms excludes the remaining admissible summands.

\section{Proof of Theorem \ref{thm:trans1}}
\label{sec:proof1}

For $9\le n\le 11$, the transversal invariance of each cone $C(b)$ is proved in Section 3 of
\cite{Chen2024} (Theorem 3.3 there). Besides the three scalar facts emphasized there (the
monotonicity of $g,h$, the bounds for $\rho$, and the inequalities of Lemma 3.6), the proof also
uses the positivity of the coefficient in \eqref{eq:remaining-coeff} and the comparison in
\eqref{eq:chen-C} below. We verify these facts at $n=8$ (Subsection \ref{subsec:proof1-8})
and at $n=7$ (Subsection \ref{subsec:proof1-7}). Together with the
dimension-free remainder of the argument of \cite{Chen2024}, this proves Theorem \ref{thm:trans1}.

\subsection{The case \texorpdfstring{$n=8$}{n=8}}
\label{subsec:proof1-8}

Throughout this subsection $n=8$, $b_{max}=\frac1{18}$ and $0<b\le b_{max}$. Substituting $n=8$
into \eqref{data} gives
\begin{equation*}
    a=\frac{2b(1+3b)^2}{2+5b},
\end{equation*}
and with
\begin{equation*}
    q:=42b^2+21b+2,\qquad f:=216b^3+144b^2+29b+2,
\end{equation*}
we have $1+12a=\frac{f}{2+5b}$, $1+14a=\frac{(1+6b)q}{2+5b}$ and
$\rho=\frac{3b(252b^2+140b+19)}{32q}$.

\begin{lemma}[Lemma 3.4 of \cite{Chen2024} at $n=8$]\label{lem:monotonicity-8}
    Let $n=8$ and $0<b\le b_{max}$. Then the two functions
    \begin{equation*}
        g(b)=\frac{(1+2(n-2)a)^2}{1+2(n-1)a}\cdot\frac{2+(n-3)b}{1+(n-2)b}
        \qquad\text{and}\qquad
        h(b)=\frac{(1+2(n-1)a)^2}{(1+2(n-2)a)(1+(n-2)b)^2}
    \end{equation*}
    are both strictly increasing in $b$.
\end{lemma}

\begin{proof}
    Direct substitution gives
    \begin{equation*}
        g(b)=\frac{f(b)^2}{(1+6b)^2q(b)},\qquad h(b)=\frac{q(b)^2}{(2+5b)f(b)},
    \end{equation*}
    and differentiation and factorization yield
    \begin{equation*}
        g'(b)=\frac{f(b)\bigl(108864b^5+117936b^4+48672b^3+9450b^2+837b+26\bigr)}
        {(1+6b)^3q(b)^2}>0,
    \end{equation*}
    \begin{equation*}
        h'(b)=\frac{4(1+4b)q(b)\bigl(189b^3+174b^2+60b+8\bigr)}{(2+5b)^2f(b)^2}>0.
    \end{equation*}
    All coefficients occurring in the numerators and denominators are positive, so $g'(b)>0$ and
    $h'(b)>0$ for all $b>0$. For later use we record the endpoint values
    \begin{equation*}
        f\Bigl(\frac1{18}\Bigr)=\frac{221}{54},\qquad q\Bigl(\frac1{18}\Bigr)=\frac{89}{27},\qquad
        g\Bigl(\frac1{18}\Bigr)=\frac{48841}{17088},\qquad
        h\Bigl(\frac1{18}\Bigr)=\frac{31684}{27183}.
    \end{equation*}
\end{proof}

\begin{lemma}[Lemma 3.5 of \cite{Chen2024} at $n=8$]\label{lem:gapcond-8}
    Let $n=8$ and $0<b\le b_{max}$. Then
    \begin{equation*}
        \rho'(b)\ge\rho'\Bigl(\frac1{18}\Bigr)=\frac{5625}{7921}>\frac49,\qquad
        0<\rho(b)<b,\qquad
        \rho\Bigl(\frac1{18}\Bigr)=\frac{31}{712}.
    \end{equation*}
\end{lemma}

\begin{proof}
    Differentiating the closed form of $\rho$ gives
    \begin{equation*}
        \rho'(b)=\frac{3\bigl(5292b^4+5292b^3+1827b^2+280b+19\bigr)}{16q(b)^2},
        \qquad
        \rho''(b)=-\frac{21\bigl(252b^2+126b+17\bigr)}{8q(b)^3}<0.
    \end{equation*}
    Hence $\rho'$ is decreasing, so $\rho'(b)\ge\rho'(\frac1{18})=\frac{5625}{7921}>\frac49$.
    Since $\rho(0)=0$ and $\rho'>0$ we have $\rho>0$, and
    \begin{equation*}
        b-\rho(b)=\frac{7b\bigl(84b^2+36b+1\bigr)}{32q(b)}>0,
    \end{equation*}
    so $\rho<b$. Direct calculation shows that $\rho(\frac1{18})=\frac{31}{712}$.
\end{proof}

\begin{lemma}[Lemma 3.6 of \cite{Chen2024} at $n=8$]\label{lem:gapcond2-8}
    Let $n=8$ and $0<b\le b_{max}$. Then, with $a,\gamma,\omega,A,P,Q$ as in \eqref{data},
    \begin{equation}\label{eq:three-8}
        (1+b\sqrt{n-2})^2<1+\frac{A}{2}\sqrt{n(n-2)},
    \end{equation}
    \begin{equation}\label{eq:three-8b}
        \left(1+\frac{A}{2}\sqrt{n(n-2)}\right)^2\frac{1}{P+nQ}<\frac{\omega}{4},
    \end{equation}
    and
    \begin{equation}\label{eq:three-8c}
        \left(1+\frac{A}{2}\sqrt{n(n-2)}\right)\frac{1}{P}<\frac{\omega}{4}.
    \end{equation}
\end{lemma}

\begin{proof}
    Write $\Lambda:=\frac{A}{2}\sqrt{48}=2\sqrt3\,A$, where
    $A=\frac{756b^3+594b^2+125b+2}{14(2+5b)}$. Also,
    \begin{equation*}
        P=\frac{2(1+6b)^2}{1+14a},\qquad P+8Q=2(1+14a).
    \end{equation*}
    Define
    \begin{equation*}
        R(b):=\frac{1-\varepsilon}{2}\sqrt{\frac{27b(2+6b)}{8\cdot 8^2\,\rho(b)^3}},
    \end{equation*}
    so that factoring $\frac{1+6b}{2+5b}$ out of $\omega$ gives
    \begin{equation}\label{eq:omegaR8}
        \frac{\omega}{4}=R(b)\,\frac{1+6b}{2(2+5b)}.
    \end{equation}

    For \eqref{eq:three-8}, the bound
    \begin{equation*}
        A-4b=\frac{756b^3+314b^2+13b+2}{14(2+5b)}>0
    \end{equation*}
    gives $\Lambda=2\sqrt3\,A>2\sqrt{48}\,b>12b$, and for $b\in(0,\frac1{18}]$ we have
    \begin{equation*}
        1+12b-(1+b\sqrt6)^2=b\bigl(12-2\sqrt6-6b\bigr)>0.
    \end{equation*}
    Thus $(1+b\sqrt6)^2<1+12b<1+\Lambda$, which is \eqref{eq:three-8}.

    For \eqref{eq:three-8b} and \eqref{eq:three-8c}, observe that
    \begin{equation*}
        \frac{2(1+4a)}{28}+4a-A=\frac{9b^2(18b+7)}{7(2+5b)}\ge0,
    \end{equation*}
    and hence
    \begin{equation}\label{eq:Lambda8bound}
        \Lambda=2\sqrt3\,A
        <\frac72\Bigl(\frac{2(1+4a)}{28}+4a\Bigr)
        =\frac54(1+12a)-1,
    \end{equation}
    so that $1+\Lambda<\frac54(1+12a)$.
    Using \eqref{eq:omegaR8} and the identities for $P$ and $P+8Q$ above, the inequalities
    \eqref{eq:three-8b} and \eqref{eq:three-8c} are respectively equivalent to
    \begin{equation}\label{eq:three-8-red}
        \frac{(1+\Lambda)^2(2+5b)}{(1+14a)(1+6b)}<R(b),
        \qquad
        \frac{(1+\Lambda)(1+14a)(2+5b)}{(1+6b)^3}<R(b).
    \end{equation}
    By \eqref{eq:Lambda8bound}, the left-hand sides of \eqref{eq:three-8-red} are bounded above by
    $\bigl(\frac54\bigr)^2g(b)$ and $\frac54\,g(b)h(b)$ respectively, which are increasing by Lemma
    \ref{lem:monotonicity-8}. On the other hand,
    \begin{equation*}
        \frac{d}{db}\left(\frac{b(2+6b)}{\rho(b)^3}\right)
        =-\frac{65536\,q(b)^2\bigl(31752b^5+49392b^4+28014b^3+7686b^2+1115b+76\bigr)}
        {27b^3\bigl(252b^2+140b+19\bigr)^4}<0,
    \end{equation*}
    so $R$ is decreasing, and it suffices to check the right endpoint $b=\frac1{18}$. There,
    \begin{equation*}
        \left(\frac54\right)^2g\Bigl(\frac1{18}\Bigr)=\frac{1221025}{273408},
        \qquad
        \frac54\,g\Bigl(\frac1{18}\Bigr)h\Bigl(\frac1{18}\Bigr)=\frac{98345}{23616},
        \qquad
        R\Bigl(\frac1{18}\Bigr)^2=\frac{4924918368783}{238328000000},
    \end{equation*}
    and all quantities being positive, squaring gives the exact margins
    \begin{equation*}
        R\Bigl(\frac1{18}\Bigr)^2-\left[\left(\frac54\right)^2g\Bigl(\frac1{18}\Bigr)\right]^2
        =\frac{25045842231370311401}{34795857494016000000}>0,
    \end{equation*}
    \begin{equation*}
        R\Bigl(\frac1{18}\Bigr)^2-\left[\frac54\,g\Bigl(\frac1{18}\Bigr)h\Bigl(\frac1{18}\Bigr)\right]^2
        =\frac{862611084879662129}{259607830464000000}>0.
    \end{equation*}
    The opposite monotonicities now prove \eqref{eq:three-8-red} throughout $(0,\frac1{18}]$, and
    hence \eqref{eq:three-8b} and \eqref{eq:three-8c}.
\end{proof}

\subsection{The case \texorpdfstring{$n=7$}{n=7}}
\label{subsec:proof1-7}

Throughout this subsection $n=7$. By Definition \ref{def:pinchingCones1}, the data
\eqref{data} use $\varepsilon=\varepsilon_7:=10^{-4}$ on the range
\begin{equation*}
    0<b\le\beta_7:=b_{max,2}=\frac{131}{2000}.
\end{equation*}
With the abbreviations
\begin{equation*}
    B:=15b^2+9b+1,\qquad C:=125b^3+100b^2+24b+2,
\end{equation*}
the data reduce to
\begin{equation}\label{eq:data7}
    \begin{split}
        a&=\frac{b(2+5b)^2}{4(1+2b)},\qquad
        \gamma=\frac{b}{2(1+2b)},\qquad
        \rho=\frac{b\bigl(375b^2+255b+43\bigr)}{49B},\\
        A&=\frac{1125b^3+1100b^2+294b+7}{63(1+2b)},\qquad
        P=\frac{2(1+2b)(1+5b)}{B},\\
        Q&=\frac{2b(3b+1)(5b+1)(5b+2)(15b+7)}{7(1+2b)B},\qquad
        \omega=(1-\varepsilon_7)\sqrt{\frac{27b(2+5b)(1+5b)^2}{392\rho^3(2+4b)^2}}.
    \end{split}
\end{equation}

\begin{lemma}[Lemma 3.4 of \cite{Chen2024} at $n=7$]\label{lem:monotonicity-7}
    Let $n=7$ and $0<b\le\beta_7$. Then $g$ and $h$ are both strictly increasing in $b$.
\end{lemma}

\begin{proof}
    With
    \begin{equation*}
        X:=1+10a=\frac{C}{2(1+2b)},\qquad Y:=1+12a=\frac{(1+5b)B}{1+2b},
    \end{equation*}
    we have
    \begin{equation*}
        g=\frac{X^2}{Y}\,\frac{2+4b}{1+5b}=\frac{C^2}{2(1+5b)^2B},
        \qquad
        h=\frac{Y^2}{X(1+5b)^2}=\frac{2B^2}{(1+2b)C}.
    \end{equation*}
    Direct differentiation and factorization give
    \begin{equation*}
        g'(b)=\frac{C\bigl(18750b^5+24375b^4+12025b^3+2770b^2+286b+10\bigr)}
        {2(1+5b)^3B^2}>0,
    \end{equation*}
    \begin{equation*}
        h'(b)=\frac{2B\bigl(375b^4+515b^3+285b^2+76b+8\bigr)}
        {(1+2b)^2C^2}>0,
    \end{equation*}
    and all factors are positive for $b>0$.
\end{proof}

\begin{lemma}[Lemma 3.5 of \cite{Chen2024} at $n=7$]\label{lem:gapcond-7}
    Let $n=7$ and $0<b\le\beta_7$. Then
    \begin{equation*}
        \rho'(b)>\frac49,\qquad 0<\rho(b)<b.
    \end{equation*}
\end{lemma}

\begin{proof}
    Differentiating the closed form of $\rho$ in \eqref{eq:data7} gives
    \begin{equation*}
        \rho'(b)=\frac{5625b^4+6750b^3+2775b^2+510b+43}{49B^2},
    \end{equation*}
    and more precisely
    \begin{equation*}
        \rho'(b)-\frac49=\frac{6525b^4+7830b^3+3219b^2+1062b+191}{441B^2}>0.
    \end{equation*}
    Moreover
    \begin{equation*}
        \rho=\frac{b\bigl(375b^2+255b+43\bigr)}{49B}>0,
        \qquad
        b-\rho=\frac{6b\bigl(60b^2+31b+1\bigr)}{49B}>0.
    \end{equation*}
\end{proof}

\begin{lemma}[Lemma 3.6 of \cite{Chen2024} at $n=7$]\label{lem:gapcond2-7}
    Let $n=7$ and $0<b\le\beta_7$. Then, with $\Lambda:=\frac{A}{2}\sqrt{35}$,
    \begin{equation}\label{eq:three-7}
        (1+b\sqrt5)^2<1+\Lambda,
    \end{equation}
    \begin{equation}\label{eq:three-7b}
        \frac{(1+\Lambda)^2}{P+7Q}<\frac{\omega}{4},
    \end{equation}
    and
    \begin{equation}\label{eq:three-7c}
        \frac{1+\Lambda}{P}<\frac{\omega}{4}.
    \end{equation}
\end{lemma}

\begin{proof}
    For \eqref{eq:three-7} we use the rational bounds
    \begin{equation*}
        \frac{1479}{250}<\sqrt{35},\qquad \sqrt5<\frac{22361}{10000},
    \end{equation*}
    verified by squaring:
    $\bigl(\frac{1479}{250}\bigr)^2-35=-\frac{59}{62500}<0$ and
    $\bigl(\frac{22361}{10000}\bigr)^2-5=\frac{14321}{10^8}>0$. The first gives
    $\Lambda=\frac{\sqrt{35}}2A>\frac{1479}{500}A$, and a direct simplification gives
    \begin{equation*}
        1+\frac{1479}{500}A-\left(1+\frac{22361}{10000}b\right)^2
        =\frac{690200000+19596780000b+79176459259b^2+89924398518b^3}{2100000000(1+2b)}>0,
    \end{equation*}
    so that
    \begin{equation*}
        (1+b\sqrt5)^2<\left(1+\frac{22361}{10000}b\right)^2<1+\frac{1479}{500}A<1+\Lambda,
    \end{equation*}
    which is \eqref{eq:three-7}.

    For \eqref{eq:three-7b} and \eqref{eq:three-7c}, define
    \begin{equation*}
        R(b):=\frac{1-\varepsilon_7}{2}\sqrt{\frac{27b(2+5b)}{392\,\rho(b)^3}},
    \end{equation*}
    so that factoring $\frac{1+5b}{2+4b}$ out of $\omega$ gives
    \begin{equation}\label{eq:omegaR7}
        \frac{\omega}{4}=R(b)\,\frac{1+5b}{2(2+4b)}.
    \end{equation}
    The sign of $(R^2)'$ is the sign of $(2+10b)\rho-3b(2+5b)\rho'$, and by Lemma
    \ref{lem:gapcond-7} ($\rho<b$ and $\rho'>\frac49$) this is strictly smaller than
    \begin{equation*}
        b(2+10b)-\frac43b(2+5b)=-\frac{2b}{3}(1-5b)<0,
    \end{equation*}
    since $\beta_7<\frac15$. Hence $R$ is strictly decreasing.

    Next, for \eqref{eq:three-7b} and \eqref{eq:three-7c} we bound $\Lambda$ from above: since
    $\bigl(\frac{59161}{10000}\bigr)^2-35=\frac{23921}{10^8}>0$,
    \begin{equation}\label{eq:Lambda7bound}
        \Lambda=\frac{\sqrt{35}}2A<\frac{59161}{20000}A=:cA,
    \end{equation}
    so with $U:=1+cA$ one has $1+\Lambda<U$. Put $Y:=1+12a$ and introduce the sharper comparison functions
    \begin{equation*}
        S_2(b):=U^2\frac{2+4b}{Y(1+5b)},\qquad
        S_3(b):=U\,\frac{Y(2+4b)}{(1+5b)^3},
    \end{equation*}
    in place of the coarser comparison through $g$ and $h$ used in \cite{Chen2024}, which is
    insufficient at $n=7$ (see Remark \ref{rem:n7reduction} below). With
    $D:=66556125b^3+65077100b^2+19913334b+1674127$, exact simplification gives
    \begin{equation*}
        S_2=\frac{D^2}{793800000000(1+5b)^2B},
    \end{equation*}
    \begin{equation*}
        S_2'=\frac{D\bigl(9983418750b^5+12978444375b^4+6219930525b^3+1408851320b^2
        +163297451b+8018255\bigr)}{793800000000(1+5b)^3B^2}>0,
    \end{equation*}
    and
    \begin{equation*}
        S_3=\frac{BD}{630000(1+2b)(1+5b)^2},
    \end{equation*}
    \begin{equation*}
    \begin{split}
        S_3'&=\frac{19966837500b^6+38713479375b^5+30194295375b^4+12165106705b^3}{630000(1+2b)^2(1+5b)^3}\\
        &\quad+\frac{2692058587b^2+313691827b+14890953}{630000(1+2b)^2(1+5b)^3}>0,
    \end{split}
    \end{equation*}
    so $S_2$ and $S_3$ are strictly increasing. At $b=\beta_7=\frac{131}{2000}$, direct
    substitution into the displayed rational formulas gives the exact rational bounds
    \begin{equation*}
        R(\beta_7)^2>\frac{21567}{1000},\qquad
        S_2(\beta_7)^2<\frac{21529}{1000}<\frac{21567}{1000},\qquad
        S_3(\beta_7)^2<\frac{18623}{1000}<\frac{21567}{1000};
    \end{equation*}
    for instance,
    \begin{equation*}
        R(\beta_7)^2
        =\frac{13975907253106977704824239165412419}{648014963448238060589335000000000}.
    \end{equation*}
    All quantities are positive, so $S_2(\beta_7)<R(\beta_7)$ and $S_3(\beta_7)<R(\beta_7)$, and the
    opposite monotonicities give
    \begin{equation*}
        S_2(b)<R(b),\qquad S_3(b)<R(b)\qquad\text{for all }0<b\le\beta_7.
    \end{equation*}
    Finally, a direct simplification gives
    \begin{equation*}
        P+7Q=2Y,\qquad P=\frac{2(1+5b)^2}{Y},
    \end{equation*}
    and together with \eqref{eq:omegaR7} these identities show that
    \begin{equation*}
        S_2<R\iff\frac{U^2}{P+7Q}<\frac{\omega}{4},
        \qquad
        S_3<R\iff\frac{U}{P}<\frac{\omega}{4}.
    \end{equation*}
    Since $1+\Lambda<U$ by \eqref{eq:Lambda7bound}, inequalities \eqref{eq:three-7b} and \eqref{eq:three-7c} follow.
\end{proof}

\begin{remark}\label{rem:n7reduction}
    The proofs of Lemmas \ref{lem:monotonicity-8}--\ref{lem:gapcond-8} and
    \ref{lem:monotonicity-7}--\ref{lem:gapcond-7} are direct in both dimensions. For the three
    inequalities, Chen's original coarse $g,h$-based endpoint reduction works unchanged at $n=8$
    (the proof of Lemma \ref{lem:gapcond2-8}). At $n=7$ that coarse reduction is insufficient, but
    the lemma itself still holds: retaining the exact quantities $S_2$ and $S_3$, proving that they
    increase, and comparing them with the decreasing function $R$ gives a short rigorous
    replacement (the proof of Lemma \ref{lem:gapcond2-7}).
\end{remark}

The remaining coefficient used in Proposition 3.7 of \cite{Chen2024} is positive on both parameter ranges:
\begin{equation}\label{eq:remaining-coeff}
\begin{split}
    &nb^2(1-2b)-2(a-b)(1-2b+nb^2)\\
    &\qquad=\frac{b^2}{2+(n-3)b}\,
    \bigl(2+(n-8)b-2(n+2)(n-2)b^2-n(n-2)^2b^3\bigr)>0.
\end{split}
\end{equation}
Indeed, the polynomial in parentheses has derivative $-48b(18b+5)<0$ at $n=8$ and
$-525b^2-180b-1<0$ at $n=7$. Its endpoint values are, respectively,
$\frac{128}{81}>0$ and $\frac{479744163}{320000000}>0$.

The auxiliary comparison used in Proposition 3.12 of \cite{Chen2024} also holds on both ranges:
\begin{equation}\label{eq:chen-C}
    \frac{(1+4b)^2(n-5)^2(n-2)}{(n-1)^2(n-4)^2\bigl(1+\frac A2\sqrt{n(n-2)}\bigr)}
    <\frac{1+4b}{n-1}<\frac12.
\end{equation}
For the first inequality, use $A>0$ and
\begin{equation*}
    (1+4b)(n-5)^2(n-2)\le
    \begin{cases}
        66<112=(n-1)(n-4)^2,&n=8,\\
        \frac{631}{25}<54=(n-1)(n-4)^2,&n=7.
    \end{cases}
\end{equation*}
For the last inequality in \eqref{eq:chen-C}, the endpoint bounds are $\frac{11}{63}<\frac12$
at $n=8$ and $\frac{631}{3000}<\frac12$ at $n=7$.
With Lemmas \ref{lem:monotonicity-8}, \ref{lem:gapcond-8} and \ref{lem:gapcond2-8} at $n=8$,
Lemmas \ref{lem:monotonicity-7}, \ref{lem:gapcond-7} and \ref{lem:gapcond2-7} at $n=7$,
and \eqref{eq:remaining-coeff} and \eqref{eq:chen-C} in place, the dimension-free remainder of
the proof of Theorem 3.3 of \cite{Chen2024} applies in both dimensions.
This proves Theorem \ref{thm:trans1}.

\section{Proof of Theorem \ref{thm:trans2}}
\label{sec:proof2}

We prove the transversal invariance of the second family $\tilde C(b)$, following the strategy of
Section 4 of \cite{Brendle2019} as adapted in Section 4 of \cite{Chen2024}. The boundary
calculation reduces the proof to three claims (Subsection \ref{subsec:proof2-common});
the claims are then verified at $n=8$ (Subsection \ref{subsec:proof2-8}) and at $n=7$
(Subsection \ref{subsec:proof2-7}).

\subsection{The boundary calculation}
\label{subsec:proof2-common}

We first recall the pullback formulation in which the invariance of $\tilde C(b)$ is checked. Write
$a_{max}=a(b_{max,2})$ and $\gamma_{max}=\gamma(b_{max,2})$. For $0\le b\le a$,
\begin{equation}\label{eq:lab-pf2}
    l_{a,b}(R)=R+b\,\Ric(R)\wedge\mathrm{id}+\frac{2(a-b)}{n}\Scal(R)\,I,
\end{equation}
with the trace identities
\begin{equation}\label{eq:labtrace-pf2}
    \Rico(l_{a,b}(R))=(1+(n-2)b)\Rico(R),\qquad
    \Scal(l_{a,b}(R))=(1+2(n-1)a)\Scal(R).
\end{equation}
With $D_{a,b}(S):=l_{a,b}^{-1}\bigl(Q(l_{a,b}(S))\bigr)-Q(S)$, a curvature operator
$R=l_{a,b}(S)$ solves $\dt R=Q(R)$ if and only if $S$ solves the pullback equation
\begin{equation}\label{eq:pullback}
    \dt S=Q(S)+D_{a,b}(S),
\end{equation}
and the B\"ohm--Wilking computation gives
\begin{equation}\label{eq:Dab-pf2}
    \begin{split}
        D_{a,b}(S)={}&\bigl(2b+(n-2)b^2-2a\bigr)\Rico(S)\wedge\Rico(S)
        +2a\,\Ric(S)\wedge\Ric(S)+2b^2\Rico(S)^2\wedge\mathrm{id}\\
        &+\frac{nb^2(1-2b)-2(a-b)(1-2b+nb^2)}{n\bigl(1+2(n-1)a\bigr)}\,
        |\Rico(S)|^2\,\mathrm{id}\wedge\mathrm{id}.
    \end{split}
\end{equation}
Moreover, if $S$ evolves by \eqref{eq:pullback}, then
\begin{equation}\label{eq:RicEvol-pf2}
    \begin{split}
        \dt\Ric(S)={}&2\bigl(S*\Ric(S)\bigr)-4b\Ric(S)^2
        +\frac4n\bigl(2b+(n-2)a\bigr)\Scal(S)\Ric(S)\\
        &+2\,\frac{n^2b^2-2(n-1)(a-b)(1-2b)}{n\bigl(1+2(n-1)a\bigr)}\,
        |\Rico(S)|^2\mathrm{id}
        +\frac{4}{n^2}(a-b)\Scal(S)^2\mathrm{id}.
    \end{split}
\end{equation}

For an orthonormal four-frame, abbreviate
\begin{equation}\label{eq:Phi}
    \Phi_\lambda(A):=A_{1313}+\lambda^2A_{1414}+A_{2323}+\lambda^2A_{2424}-2\lambda A_{1234},
\end{equation}
so that $\Phi_1\ge0$ is the weakly PIC condition, and put
\begin{equation}\label{eq:tildea}
    \tilde a(b)=b+\frac{n-2}{2}b^2,\qquad 2b+(n-2)b^2-2\tilde a=0.
\end{equation}
With this notation, condition (ii) of Definition \ref{def:pinchingCones2} reads
\begin{equation}\label{eq:Zlambda}
    Z_\lambda(S):=\Phi_\lambda(S)+\sqrt{2\tilde a}\,(1-\lambda^2)
    \bigl(\Ric(S)_{11}+\Ric(S)_{22}\bigr)\ge0
\end{equation}
for every orthonormal four-frame and every $\lambda\in[0,1]$. The term $\lambda^2S_{2424}$ in
$\Phi_\lambda$ is essential for the boundary computation below. Note that $Z_1=\Phi_1$.

We use two computational lemmas, both valid in every dimension $n\ge5$; they are stated and
proved in \cite{Brendle2019} (cf.\ \cite{Chen2024}), and we do not reproduce the proofs.

\begin{lemma}[Boundary reaction; cf.\ Proposition A.8 of \cite{Brendle2019}, Lemma 4.8 of \cite{Chen2024}]\label{lem:boundary-reaction}
    Let $n\ge5$, let $S\in\CB(\mathbb R^n)$, and let $H$ be a symmetric two-tensor. Suppose that
    \begin{equation*}
        \Phi_\lambda(S)+(1-\lambda^2)\bigl(H_{11}+H_{22}\bigr)\ge0
    \end{equation*}
    for all orthonormal four-frames and all $\lambda\in[0,1]$, with equality at a given frame and
    some $\lambda\in[0,1)$. Then
    \begin{equation}\label{eq:boundary-reaction}
        \Phi_\lambda\bigl(Q(S)\bigr)+\Phi_\lambda(H\wedge H)
        +2(1-\lambda^2)\bigl((S*H)_{11}+(S*H)_{22}\bigr)
        \ge(1+\lambda^2)\bigl(H_{11}+H_{22}\bigr)^2\ge0.
    \end{equation}
\end{lemma}

\begin{lemma}[Ricci quadratic estimate; cf.\ Lemma A.1 of \cite{Brendle2019}, Lemma 4.7 of \cite{Chen2024}]\label{lem:ricci-quadratic}
    Let $0\le\zeta\le1$ and $0<\varrho\le1$, and let $H$ be a symmetric two-tensor whose largest
    eigenvalue is at most $\frac12\mathrm{tr}\,H$ and whose two smallest eigenvalues have sum at
    least $\frac{2(1-\zeta)}{n}\mathrm{tr}\,H$. Then, with
    $H^0=H-\frac1n\mathrm{tr}(H)\,\mathrm{id}$, every orthonormal pair $\{e_1,e_2\}$ satisfies
    \begin{equation}\label{eq:ricci-quadratic}
        \frac{n-2}{n}\,\mathrm{tr}(H)\,\bigl(H_{11}+H_{22}\bigr)
        -\varrho\bigl((H^0)^2_{11}+(H^0)^2_{22}\bigr)
        \ge\frac{2}{n^2}\left((n-2)(1-\zeta)-2\zeta^2\varrho\,
        \frac{n^2-2n+2}{(n-2)^2}\right)\mathrm{tr}(H)^2.
    \end{equation}
\end{lemma}

Condition (i) of Definition \ref{def:pinchingCones2} is preserved directly under
\eqref{eq:pullback}: with $\tilde a=\tilde a(b)$,
\begin{equation}\label{eq:condi}
    \dt\,l_{\tilde a,b}(S)=l_{\tilde a,b}\bigl(Q(S)+D_{\tilde a,b}(S)\bigr)
    =Q\bigl(l_{\tilde a,b}(S)\bigr),
\end{equation}
so the first-family invariance at $b_{max,2}$ (Theorem \ref{thm:trans1}) preserves, and at its
boundary strictly improves, condition (i). In particular $\Scal(S)>0$ for $S\ne0$, since the
first-family cone has positive scalar curvature and $l_{\tilde a,b}$ rescales scalar curvature by
the positive factor $1+2(n-1)\tilde a$.

It remains to prove the strict inward reaction of condition (ii). Fix a boundary datum
$Z_\lambda(S)=0$. Substituting $a=\tilde a$ into \eqref{eq:Dab-pf2} and using \eqref{eq:tildea},
the $\Rico\wedge\Rico$ term disappears:
\begin{equation}\label{eq:Dab-tilde}
    D_{\tilde a,b}(S)=2\tilde a\,\Ric\wedge\Ric+2b^2\Rico^2\wedge\mathrm{id}
    +c_4\,|\Rico|^2\,\mathrm{id}\wedge\mathrm{id},
\end{equation}
where
\begin{equation}\label{eq:c4}
    c_4=\frac{nb^2(1-2b)-2(\tilde a-b)(1-2b+nb^2)}{n\bigl(1+2(n-1)\tilde a\bigr)}.
\end{equation}
\begin{claim}[Discards]\label{claim:discards}
    For $n=7,8$ and $0<b\le\tilde b_{max,2}$: $c_4>0$ in \eqref{eq:c4},
    and the coefficient of $|\Rico|^2\mathrm{id}$ in \eqref{eq:RicEvol-pf2} is positive.
\end{claim}
Consequently the terms
\begin{equation*}
    2b^2\Rico(S)^2\wedge\mathrm{id},\qquad
    c_4\,|\Rico(S)|^2\,\mathrm{id}\wedge\mathrm{id}
\end{equation*}
have nonnegative $\Phi_\lambda$-reaction, and the $|\Rico|^2\mathrm{id}$ term in
\eqref{eq:RicEvol-pf2} has nonnegative $Z_\lambda$-reaction. They may therefore be discarded in a
lower bound.

\medskip\noindent\textbf{The case $0\le\lambda<1$.} Differentiating \eqref{eq:Zlambda}, inserting
\eqref{eq:RicEvol-pf2} and \eqref{eq:Dab-tilde}, and discarding only the nonnegative
contributions just described gives
\begin{equation}\label{eq:Zdot}
    \begin{split}
        \dt Z_\lambda(S)\ge{}&
        \Bigl[\Phi_\lambda\bigl(Q(S)\bigr)+2\tilde a\,\Phi_\lambda(\Ric\wedge\Ric)
        +2\sqrt{2\tilde a}\,(1-\lambda^2)\bigl((S*\Ric)_{11}+(S*\Ric)_{22}\bigr)\Bigr]\\
        &+\sqrt{2\tilde a}\,(1-\lambda^2)\Bigl[
        \frac{4(n-2)\tilde a}{n}\,\Scal\,\bigl(\Ric_{11}+\Ric_{22}\bigr)
        -4b\bigl((\Rico)^2_{11}+(\Rico)^2_{22}\bigr)
        +\frac{8\tilde a}{n^2}\,\Scal^2\Bigr].
    \end{split}
\end{equation}
Here the displayed coefficients are exact: writing
$(\Ric^2)_{ii}=(\Rico^2)_{ii}+\frac2n\Scal\Ric_{ii}-\frac1{n^2}\Scal^2$ in
\eqref{eq:RicEvol-pf2}, the coefficient of $\Scal\,\bigl(\Ric_{11}+\Ric_{22}\bigr)$ is
$\frac4n\bigl(2b+(n-2)\tilde a\bigr)-\frac{8b}{n}=\frac{4(n-2)\tilde a}{n}$, and the
scalar-square coefficient is
$2\cdot\frac{4}{n^2}(\tilde a-b)+2\cdot\frac{4b}{n^2}=\frac{8\tilde a}{n^2}$.

The first bracket in \eqref{eq:Zdot} is nonnegative by Lemma \ref{lem:boundary-reaction} applied
with $H=\sqrt{2\tilde a}\,\Ric(S)$: its hypothesis is exactly $Z_\lambda\ge0$ for all frames and
all $\lambda\in[0,1]$, together with the assumed equality $Z_\lambda(S)=0$.

For the second bracket, condition (ii) at $\lambda=1$ gives $S\in\CPIC$, and Lemma \ref{lem:wpic}(1) gives $\lambda_{\max}\bigl(\Ric(S)\bigr)\le\frac12\Scal(S)$. Moreover the
first-cone control, transferred by condition (i), gives
\begin{equation}\label{eq:zeta}
    \Ric(S)_{11}+\Ric(S)_{22}\ge\frac{2\bigl(1-\zeta(b)\bigr)}{n}\Scal(S)
\end{equation}
for every orthonormal pair, where
\begin{equation}\label{eq:zeta-def}
    \zeta(b)=\frac{1+2(n-1)\tilde a(b)}{1+2(n-1)a_{max}}\,
    \frac{1+(n-2)b_{max,2}}{1+(n-2)b}\,\bigl(1+\gamma_{max}\bigr).
\end{equation}
Indeed, writing $l_{\tilde a,b}(S)=l_{a_{max},b_{max,2}}(U)$ with
$U\in\mathcal{E}(b_{max,2})$,
condition (3) of Definition \ref{def:pinchingCones1} gives
$\Ric(U)_{11}+\Ric(U)_{22}+\frac{2\gamma_{max}}{n}\Scal(U)\ge0$, and pulling this back through
the two transformations by means of \eqref{eq:labtrace-pf2} yields \eqref{eq:zeta}.

\begin{claim}[Two-eigenvalue gap]\label{claim:zeta}
    For $n=7,8$ and $0<b\le\tilde b_{max,2}$: $\zeta(b)<1$.
\end{claim}

Since every factor in \eqref{eq:zeta-def} is positive, Claim \ref{claim:zeta} gives
$\zeta(b)\in(0,1)$. Moreover,
\begin{equation*}
    \varrho=\frac{b}{\tilde a}=\Bigl(1+\frac{n-2}{2}b\Bigr)^{-1}\in(0,1],
    \qquad
    K:=\frac{n^2-2n+2}{(n-2)^2}.
\end{equation*}
Thus Lemma \ref{lem:ricci-quadratic} applies with $H=\Ric(S)$.
After multiplication by $4\tilde a$ it yields
\begin{equation}\label{eq:secondbracket}
    \begin{split}
        &\frac{4(n-2)\tilde a}{n}\,\Scal\,\bigl(\Ric_{11}+\Ric_{22}\bigr)
        -4b\bigl((\Rico)^2_{11}+(\Rico)^2_{22}\bigr)\\
        &\hspace{4em}\ge\frac{8\tilde a}{n^2}
        \left((n-2)(1-\zeta)-2K\zeta^2\frac{b}{\tilde a}\right)\Scal^2.
    \end{split}
\end{equation}
Adding the remaining $\frac{8\tilde a}{n^2}\Scal^2$ from \eqref{eq:Zdot} proves
\begin{equation}\label{eq:ZdotF}
    \dt Z_\lambda(S)\ge\sqrt{2\tilde a}\,(1-\lambda^2)\frac{8\tilde a}{n^2}\,
    F(b)\,\Scal(S)^2,
\end{equation}
where
\begin{equation}\label{eq:Fb}
    F(b):=1+(n-2)\bigl(1-\zeta(b)\bigr)-2K\zeta(b)^2\frac{b}{\tilde a}
    =1+(n-2)\bigl(1-\zeta(b)\bigr)-2K\,\frac{\zeta(b)^2}{1+\frac{n-2}{2}b}.
\end{equation}
\begin{claim}[Positivity of $F$]\label{claim:F}
    For $n=7,8$ and $0<b\le\tilde b_{max,2}$: $F(b)>0$.
\end{claim}
Since $\lambda<1$, $\tilde a>0$ and, for $S\ne0$, $\Scal(S)>0$, the right-hand side of
\eqref{eq:ZdotF} is strictly positive: every nonzero boundary point with
$\lambda\in[0,1)$ is strictly pushed into the cone.

\medskip\noindent\textbf{The case $\lambda=1$.} The Ricci correction in $Z_1$ vanishes, and from
\eqref{eq:Dab-tilde}, after discarding the same nonnegative terms,
\begin{equation}\label{eq:Zdot1}
    \dt Z_1(S)\ge\Phi_1\bigl(Q(S)\bigr)+2\tilde a\,(\Ric\wedge\Ric)(\varphi,\bar\varphi),
    \qquad \varphi=(e_1+ie_2)\wedge(e_3+ie_4).
\end{equation}
The first term is nonnegative because $\CPIC$ is preserved by the Hamilton ODE and
$Z_1(S)=S(\varphi,\bar\varphi)=0$. For the second, write $\varphi=z\wedge w$. A unitary change
of basis in $\operatorname{span}_{\mathbb C}\{z,w\}$ diagonalizes the Hermitian form
$\Ric(S)(\cdot,\overline{\cdot})$ and changes $z\wedge w$ only by a unit complex factor.
Relabelling the resulting basis as $z=e_1+ie_2$, $w=e_3+ie_4$, we have
$\Ric(S)(z,\bar w)=0$ and
\begin{equation}\label{eq:RicwedgeRic}
    (\Ric\wedge\Ric)(z,w,\bar z,\bar w)
    =2\bigl(\Ric_{11}+\Ric_{22}\bigr)\bigl(\Ric_{33}+\Ric_{44}\bigr).
\end{equation}
Claim \ref{claim:zeta} and \eqref{eq:zeta} show that, for $S\ne0$, both factors are strictly
positive. Hence $\dt Z_1(S)>0$ at every nonzero boundary point.

The three claims introduced above are the only dimension-specific inputs of the boundary
calculation; they are verified in Subsections \ref{subsec:proof2-8} and \ref{subsec:proof2-7}.

\subsection{Proof of the claims for \texorpdfstring{$n=8$}{n=8}}
\label{subsec:proof2-8}

Let
\begin{equation*}
    0<b\le\tilde b_{max}=\frac1{40},\qquad \tilde a=b+3b^2,\qquad b_{max}=\frac1{18}.
\end{equation*}
For the first-family endpoint,
\begin{equation*}
    a_{max}=\frac{49}{738},\qquad \gamma_{max}=\frac1{41},
\end{equation*}
and the quantity \eqref{eq:zeta-def} is
\begin{equation}\label{eq:zeta8}
    \zeta(b)=\frac{1+14\tilde a(b)}{1+14a_{max}}\,\frac{1+6b_{max}}{1+6b}\,(1+\gamma_{max})
    =\frac{63}{89}\,\frac{1+14b+42b^2}{1+6b}.
\end{equation}
\medskip\noindent\textbf{Proof of Claim \ref{claim:discards}.}
The exact scalar identities underlying the discards in \eqref{eq:Zdot} are
\begin{equation}\label{eq:discard8}
    8b^2(1-2b)-2(\tilde a-b)(1-2b+8b^2)=2b^2(1+4b)(1-6b)>0,
\end{equation}
\begin{equation}\label{eq:discard8b}
    64b^2-14(\tilde a-b)(1-2b)=2b^2(11+42b)>0:
\end{equation}
the first makes $c_4>0$ in \eqref{eq:c4}, and the second makes the $|\Rico|^2\mathrm{id}$
coefficient in \eqref{eq:RicEvol-pf2} positive.

\medskip\noindent\textbf{Proof of Claim \ref{claim:zeta}.}
\begin{equation}\label{eq:zeta8mono}
    \zeta'(b)=\frac{252\bigl(2+21b+63b^2\bigr)}{89(1+6b)^2}>0,
    \qquad
    \zeta(b)\le\zeta\Bigl(\frac1{40}\Bigr)=\frac{69363}{81880}<1.
\end{equation}

\medskip\noindent\textbf{Proof of Claim \ref{claim:F}.}
Since
\begin{equation*}
    K=\frac{8^2-16+2}{6^2}=\frac{25}{18},
\end{equation*}
the function \eqref{eq:Fb} has the exact expression
\begin{equation}\label{eq:F8}
    F(b)=\frac{10780-250761b-5352228b^2-28171584b^3-44881452b^4}{7921(1+3b)(1+6b)^2}.
\end{equation}
Its derivative is
\begin{equation*}
    F'(b)=-\frac{63(12823272b^5+14960484b^4+6686424b^3+1453080b^2+155270b+6547)}
    {7921(1+3b)^2(1+6b)^3}<0,
\end{equation*}
so its endpoint value gives
\begin{equation}\label{eq:F8end}
    F(b)\ge F\Bigl(\frac1{40}\Bigr)=\frac{453196597}{7207159480}>0.
\end{equation}
The three claims are now verified for $n=8$. By the boundary calculation of Subsection
\ref{subsec:proof2-common}, $\tilde{\mathcal{E}}(b)$ is transversally invariant under
\eqref{eq:pullback}, and by \eqref{eq:condi} the cone
$\tilde C(b)$ is transversally invariant under $\dt R=Q(R)$ for $0<b\le\frac1{40}$.

\begin{remark}\label{rem:sharpbracket}
    It is important that \eqref{eq:F8} uses the factor $\frac{b}{\tilde a}=(1+3b)^{-1}$. The
    stronger estimate $1+6(1-\zeta)\ge2K\zeta^2$, in which this factor is replaced by $1$ (as in
    Lemma 4.3 of \cite{Chen2024}), is false at $n=8$: at $b=\frac1{40}$, its left-hand side minus
    its right-hand side equals
    \begin{equation*}
        -\frac{102165973}{1340866880}<0.
    \end{equation*}
    The proof does not need it: the two terms in \eqref{eq:secondbracket} occur together in the
    evolution with the single factor $\frac{b}{\tilde a}$, and the actual quantity $F(b)$ is
    strictly positive by \eqref{eq:F8end}.
\end{remark}

\subsection{Proof of the claims for \texorpdfstring{$n=7$}{n=7}}
\label{subsec:proof2-7}

Let
\begin{equation*}
    0<b\le\frac3{125},\qquad \tilde a=b+\frac52b^2,\qquad b_{max}=\frac{131}{2000}.
\end{equation*}
Here
\begin{equation*}
    a_{max}=\frac{113545691}{1447680000},\qquad \gamma_{max}=\frac{131}{4524},
\end{equation*}
and the quantity \eqref{eq:zeta-def} is
\begin{equation}\label{eq:zeta7}
    \zeta(b)=\frac{1+12\tilde a(b)}{1+12a_{max}}\,\frac{1+5b_{max}}{1+5b}\,(1+\gamma_{max})
    =\frac{931000}{1323083}\,\frac{1+12b+30b^2}{1+5b}.
\end{equation}
\medskip\noindent\textbf{Proof of Claim \ref{claim:discards}.}
The exact scalar identities are
\begin{equation}\label{eq:discard7}
    7b^2(1-2b)-2(\tilde a-b)(1-2b+7b^2)=b^2\bigl(2-4b-35b^2\bigr)>0,
\end{equation}
\begin{equation}\label{eq:discard7b}
    49b^2-12(\tilde a-b)(1-2b)=b^2\bigl(19+60b\bigr)>0;
\end{equation}
the polynomial $2-4b-35b^2$ is decreasing and equals $\frac{5887}{3125}>0$ at $b=\frac3{125}$,
which verifies the strict inequality \eqref{eq:discard7} throughout the stated interval, and the
first identity makes $c_4>0$ in \eqref{eq:c4} while the second makes the $|\Rico|^2\mathrm{id}$
coefficient in \eqref{eq:RicEvol-pf2} positive, exactly as in dimension eight.

\medskip\noindent\textbf{Proof of Claim \ref{claim:zeta}.}
We have
\begin{equation}\label{eq:1minuszeta7}
    1-\zeta(b)=\frac{392083-4556585b-27930000b^2}{1323083(1+5b)},
\end{equation}
and direct differentiation gives
\begin{equation*}
    \zeta'(b)=\frac{931000(7+60b+150b^2)}{1323083(1+5b)^2}>0.
\end{equation*}
Hence
\begin{equation}\label{eq:zeta7end}
    \zeta(b)\le\zeta\Bigl(\frac3{125}\Bigr)=\frac{1085014}{1323083}<1.
\end{equation}

\medskip\noindent\textbf{Proof of Claim \ref{claim:F}.}
In this dimension
\begin{equation*}
    K=\frac{7^2-14+2}{5^2}=\frac{37}{25},\qquad
    \frac{b}{\tilde a}=\frac{2}{2+5b}\in(0,1],
\end{equation*}
the exact function from \eqref{eq:Fb} is
\begin{equation}\label{eq:F7}
    F(b)=\frac{2\bigl(48073724982-1361705572725b-22007606233400b^2-94605191101125b^3
    -124828596375000b^4\bigr)}{47312124997(1+5b)^2(2+5b)}.
\end{equation}
Writing
\begin{equation*}
    p_7(b)=1340801250b^5+1877121750b^4+1005603625b^3
    +261905005b^2+33520472b+1686468,
\end{equation*}
we have
\begin{equation*}
    F'(b)=-\frac{4655000p_7(b)}{47312124997(1+5b)^3(2+5b)^2}<0.
\end{equation*}
Thus $F$ satisfies the exact endpoint bound
\begin{equation}\label{eq:F7end}
    F(b)\ge F\Bigl(\frac3{125}\Bigr)=\frac{54495087372}{2507542624841}>0.
\end{equation}

The three claims are now verified for $n=7$, and the remaining steps of the argument of
Subsection \ref{subsec:proof2-common} are dimension-free: equation \eqref{eq:tildea} reads
$2b+5b^2-2\tilde a=0$, so the decomposition \eqref{eq:Dab-tilde} holds, and the boundary
reaction (Lemma \ref{lem:boundary-reaction}), the pullback identity \eqref{eq:condi} and the
preservation of condition (i) require no changes. Hence $\tilde{\mathcal{E}}(b)$ is
transversally invariant under \eqref{eq:pullback}, and $\tilde C(b)$ is transversally invariant
under $\dt R=Q(R)$ for $0<b\le\frac3{125}$.

Together with Subsection \ref{subsec:proof2-8}, this proves Theorem \ref{thm:trans2}.

\section{Proof of Theorem \ref{thm:gluing}}
\label{sec:proof3}

The gluing equality is reduced, by a dimension-free shift computation, to a single claim
(Subsection \ref{subsec:proof3-reduction}). The claim is then proved at $n=8$ (Subsection
\ref{subsec:proof3-8}) by exhibiting an explicit
nonnegative combination of averaged frame inequalities and applying a Bernstein representation in
$\lambda$.

Throughout this section, an algebraic curvature operator is written $T$, and
$T_{ijkl}:=T(e_i,e_j,e_k,e_l)$. We use the frame functional
\begin{equation}\label{eq:gl-Wcal}
    \mathcal{W}_T(v_1,v_2,v_3,v_4):=T_{1313}+T_{1414}+T_{2323}+T_{2424}-2T_{1234},
\end{equation}
so that $T$ is weakly PIC precisely when $\mathcal{W}_T(v_1,v_2,v_3,v_4)\ge0$ for every orthonormal
four-frame. Sign-averaging this inequality means averaging it after independently changing the
signs of selected frame vectors. All terms odd in one of those vectors disappear. We use the
standard tilted consequence of weakly PIC: for every $p\ge5$ and $\lambda\in[0,1]$,
\begin{equation}\label{eq:gl-tilt}
    T_{1313}+\lambda^2T_{1414}+T_{2323}+\lambda^2T_{2424}-2\lambda T_{1234}
    +(1-\lambda^2)\bigl(T_{1p1p}+T_{2p2p}\bigr)\ge0,
\end{equation}
obtained by averaging the weakly PIC inequalities of the two frames
$\{e_1,e_2,e_3,\lambda e_4\pm\sqrt{1-\lambda^2}\,e_p\}$. Finally, we use the trace consequences
\begin{equation}\label{eq:gl-trace}
    \Ric_{11}+\Ric_{22}+\Ric_{33}+\Ric_{44}\ge0,\qquad \Scal\ge0,
\end{equation}
which follow from Lemma \ref{lem:wpic}(3) and (6).

\subsection{The shift and the reduction of cone equality}
\label{subsec:proof3-reduction}

Let $b_{max}=b_{max,2}$, $a_{max}=a(b_{max})$,
$\tilde b_{max}=\tilde b_{max,2}$ and $\tilde a_{max}=\tilde a(\tilde b_{max})$ denote the
endpoint data of the first and second cone
families. Put
\begin{equation*}
    S:=l_{\tilde a_{max},\tilde b_{max}}^{-1}\bigl(l_{a_{max},b_{max}}(T)\bigr).
\end{equation*}
On the scalar, trace-free Ricci and Weyl summands in
\begin{equation*}
    \CB(\mathbb{R}^n)=\langle\mathrm{id}\wedge\mathrm{id}\rangle\oplus\langle\Rico\rangle\oplus\langle\mathcal W\rangle,
\end{equation*}
the two maps in the preceding display act by scalars, and a direct comparison of the multipliers
gives
\begin{equation}\label{eq:gl-shift}
    S=T+\kappa\,\Ric(T)\wedge\mathrm{id}+\frac Yn\Scal(T)\,\mathrm{id}\wedge\mathrm{id},
\end{equation}
where
\begin{equation}\label{eq:gl-kappa}
    \kappa=\frac{b_{max}-\tilde b_{max}}{1+(n-2)\tilde b_{max}},\qquad
    \Delta_A=\frac{a_{max}-\tilde a_{max}}{1+2(n-1)\tilde a_{max}},\qquad
    Y=\Delta_A-\kappa.
\end{equation}
Indeed, the Weyl multiplier is $1$. On $\varphi_0\wedge\mathrm{id}$, where $\Ric=(n-2)\varphi_0$
and $\Scal=0$, the two multipliers are respectively
$\frac{1+(n-2)b_{max}}{1+(n-2)\tilde b_{max}}$ and $1+(n-2)\kappa$, and on
$c\,\mathrm{id}\wedge\mathrm{id}$, where $\Ric=2(n-1)c\,\mathrm{id}$ and $\Scal=2n(n-1)c$, they
are respectively $\frac{1+2(n-1)a_{max}}{1+2(n-1)\tilde a_{max}}$ and
$1+2(n-1)(\kappa+Y)=1+2(n-1)\Delta_A$.

For an orthonormal four-frame and $\lambda\in[0,1]$, set
\begin{equation}\label{eq:gl-Llambda}
    L_\lambda(T):=S_{1313}+S_{2323}+\lambda^2\bigl(S_{1414}+S_{2424}\bigr)-2\lambda S_{1234}
    +s(1-\lambda^2)\bigl(\Ric(S)_{11}+\Ric(S)_{22}\bigr),
\end{equation}
where $s=\sqrt{2\tilde a_{max}}$. The desired reduced statement is the following claim.

\begin{claim}\label{claim:glstar}
    For $n=8$, with the endpoint data in \eqref{eq:gl-endpoint8}:
    $L_\lambda(T)\ge0$ for every weakly PIC $T$, every orthonormal four-frame, and every
    $\lambda\in[0,1]$.
\end{claim}

\paragraph{Note on dimension seven.}\label{note:dimension-seven}
The separate invariance results of Theorems \ref{thm:trans1} and \ref{thm:trans2} also hold at
$n=7$, but the seven-dimensional analogue of Claim \ref{claim:glstar} is not established here.
The averaging argument for (E2) in \eqref{eq:gl-E8} requires a four-frame in
$\mathrm{span}\{e_5,\dots,e_n\}$, which is unavailable at $n=7$. Accordingly,
Theorem \ref{thm:gluing}, Proposition \ref{prop:pinching}, and the classification theorems
\ref{thm:main1} and \ref{thm:main2} are asserted only in dimension eight.

The reduction of Theorem \ref{thm:gluing} to Claim \ref{claim:glstar} is as follows. The inclusion
$\tilde C(\tilde b_{max})\subset C(b_{max})$ is exactly the first defining condition of
Definition \ref{def:pinchingCones2}. Conversely, let
\begin{equation*}
    R=l_{a_{max},b_{max}}(T)\in C(b_{max}),\qquad T\in\mathcal{E}(b_{max}),
\end{equation*}
and define $S$ by the shift \eqref{eq:gl-shift}. Then
$l_{\tilde a_{max},\tilde b_{max}}(S)=R\in C(b_{max})$, so $S$ satisfies the first defining
condition of $\tilde{\mathcal{E}}(\tilde b_{max})$. Conditions (1) and (2) of Definition
\ref{def:pinchingCones1} imply that $T$ is weakly PIC. Since $s=\sqrt{2\tilde a_{max}}$, the
$Z_\lambda$-condition for $S$ is exactly \eqref{eq:gl-Llambda}, hence holds by Claim
\ref{claim:glstar}. Thus $S\in\tilde{\mathcal{E}}(\tilde b_{max})$, and consequently
$R=l_{\tilde a_{max},\tilde b_{max}}(S)\in\tilde C(\tilde b_{max})$. This proves the reverse
inclusion $C(b_{max})\subset\tilde C(\tilde b_{max})$. Accordingly, proving Claim
\ref{claim:glstar} proves the gluing equality.

We close this subsection with the idea behind the proof of Claim \ref{claim:glstar}, carried
out in Subsection \ref{subsec:proof3-8}. For a fixed frame,
$L_\lambda(T)$ is a quadratic polynomial in $\lambda$ whose coefficients are linear in $T$. One
part of it, the tilted-frame average, is nonnegative, and the remainder $Q(\lambda)$ is
controlled through its Bernstein representation on $[0,1]$: the control value at $1$ is
nonnegative by the un-tilted frame and trace inequalities, and the other two control values
follow from the nonnegativity of a single linear functional $K$ of $T$ (together with $x\ge0$,
arranged by reflecting $e_4$). The proof is thus reduced to $K\ge0$, which is certified by
writing $2K$ as an explicit nonnegative combination of the eight averaged frame inequalities
\eqref{eq:gl-E8}, with all multipliers strictly positive. 

\subsection{Proof of Claim \ref{claim:glstar} for \texorpdfstring{$n=8$}{n=8}}
\label{subsec:proof3-8}

In dimension $8$ the endpoint constants are
\begin{equation}\label{eq:gl-endpoint8}
    b_{max}=\frac1{18},\qquad a_{max}=\frac{49}{738},\qquad
    \tilde b_{max}=\frac1{40},\qquad \tilde a_{max}=\frac{43}{1600},\qquad
    s=\sqrt{2\tilde a_{max}}=\frac{\sqrt{86}}{40}.
\end{equation}
The shift \eqref{eq:gl-shift} has
\begin{equation}\label{eq:gl-shift8}
    \kappa=\frac{11}{414},\qquad
    \Delta_A=\frac{23333}{812538},\qquad
    Y=\frac{20054}{9344187}>0,
\end{equation}
and
\begin{equation}\label{eq:gl-Ric8}
    \Ric(S)_{11}+\Ric(S)_{22}=R\bigl(\Ric(T)_{11}+\Ric(T)_{22}\bigr)+\frac V4\Scal(T),
    \qquad \Scal(S)=(1+14\Delta_A)\Scal(T),
\end{equation}
where
\begin{equation*}
    R=1+6\kappa=\frac{80}{69},\qquad
    V=14\Delta_A-6\kappa=8\kappa+14Y=\frac{2266960}{9344187}.
\end{equation*}
For example, $\kappa=\frac{1/18-1/40}{1+6/40}=\frac{11/360}{23/20}=\frac{11}{414}$.

Fix an orthonormal frame $\{e_1,\dots,e_8\}$ and abbreviate
\begin{equation*}
\begin{split}
    \beta&=T_{1313}+T_{2323},\qquad \gamma^*=T_{1414}+T_{2424},\qquad x=T_{1234},\\
    \tau&=\frac14\sum_{p=5}^8\bigl(T_{1p1p}+T_{2p2p}\bigr),\qquad
    U=\Ric_{11}+\Ric_{22},\qquad u=\lambda^2.
\end{split}
\end{equation*}
Since $(\Ric\wedge\mathrm{id})_{ijij}=\Ric_{ii}+\Ric_{jj}$,
$(\mathrm{id}\wedge\mathrm{id})_{ijij}=2$ and $(\Ric\wedge\mathrm{id})_{1234}=0$, the shift
\eqref{eq:gl-shift8} gives
\begin{equation*}
    S_{1313}+S_{2323}=\beta+\kappa\bigl(U+2\Ric_{33}\bigr)+\frac Y2\Scal,\qquad
    S_{1414}+S_{2424}=\gamma^*+\kappa\bigl(U+2\Ric_{44}\bigr)+\frac Y2\Scal,\qquad
    S_{1234}=x.
\end{equation*}
Combining these identities with \eqref{eq:gl-Ric8} yields the exact expansion
\begin{equation}\label{eq:gl-expand8}
    L_\lambda(T)=\Phi(\lambda)+(1-u)N_0+uN_1,\qquad
    \Phi(\lambda)=\beta+u\gamma^*-2\lambda x,
\end{equation}
where
\begin{equation}\label{eq:gl-N8}
    N_0=(\kappa+sR)\,U+2\kappa\Ric_{33}+\Bigl(\frac Y2+\frac{sV}{4}\Bigr)\Scal,\qquad
    N_1=2\kappa\bigl(U+\Ric_{33}+\Ric_{44}\bigr)+Y\Scal.
\end{equation}
Write
\begin{equation}\label{eq:gl-const8}
\begin{split}
    c_1&:=\kappa+sR=\frac{11}{414}+\frac{2\sqrt{86}}{69},\qquad
    c_2:=2\kappa=\frac{11}{207},\\
    c_3&:=\frac Y2+\frac{sV}{4}=\frac{10027}{9344187}+\frac{28337\sqrt{86}}{18688374},\qquad
    d:=c_1-\frac14=\frac{2\sqrt{86}}{69}-\frac{185}{828}.
\end{split}
\end{equation}
The elementary enclosure $9.27<\sqrt{86}<9.28$ gives
\begin{equation}\label{eq:gl-encl8}
    0.0452<d<0.0456,\qquad 0.05314<c_2<0.05315,\qquad 0.01512<c_3<0.01515,
\end{equation}
and hence
\begin{equation}\label{eq:gl-ineq8}
    \begin{aligned}
        2d+2c_3-c_2&>0.0674, & 1-8c_2-20c_3&>0.2718,\\
        2c_2+12c_3-4d&>0.1053, & 8c_2+44c_3-1&>0.0904,\\
        12d+14c_2+52c_3-2&>0.0726, & 2+4d-6c_2-12c_3&>1.5.
    \end{aligned}
\end{equation}

\medskip\noindent\textbf{The frame inequalities.}
Let
$W=\mathrm{span}\{e_5,e_6,e_7,e_8\}$ and set
\begin{equation}\label{eq:gl-t8}
    \begin{gathered}
        t_{12}=T_{1212},\qquad t_3=\frac12\bigl(T_{1313}+T_{2323}\bigr)=\frac\beta2,\qquad
        t_4=\frac12\bigl(T_{1414}+T_{2424}\bigr)=\frac{\gamma^*}2,\\
        t_W=\frac18\sum_{p\ge5}\bigl(T_{1p1p}+T_{2p2p}\bigr)=\frac\tau2,\qquad t_{34}=T_{3434},\\
        t_{3W}=\frac14\sum_{p\ge5}T_{3p3p},\qquad
        t_{4W}=\frac14\sum_{p\ge5}T_{4p4p},\qquad
        t_{WW}=\frac16\sum_{5\le p<q\le8}T_{pqpq}.
    \end{gathered}
\end{equation}
These are normalized traces and therefore do not depend on the chosen basis of $W$. Direct
contraction gives
\begin{equation}\label{eq:gl-traces8}
    U=2\bigl(t_{12}+t_3+t_4+4t_W\bigr),\qquad
    \Ric_{33}=2t_3+t_{34}+4t_{3W},\qquad
    \Ric_{44}=2t_4+t_{34}+4t_{4W},
\end{equation}
\begin{equation}\label{eq:gl-scal8}
    \Scal=2t_{12}+4t_3+4t_4+16t_W+2t_{34}+8t_{3W}+8t_{4W}+12t_{WW}.
\end{equation}
The following eight quantities are nonnegative. Each is an averaged weakly PIC inequality:
\begin{equation}\label{eq:gl-E8}
    \begin{array}{rclrcl}
        \mathrm{(E1)}&t_W\ge0,&\mathrm{(E2)}&t_{WW}\ge0,\\
        \mathrm{(E3)}&t_3+t_W\ge0,&\mathrm{(E4)}&t_W+t_{4W}\ge0,\\
        \mathrm{(E5)}&t_{12}+2t_W+t_{WW}\ge0,&\mathrm{(E6)}&t_{12}+t_3+t_W+t_{3W}\ge0,\\
        \mathrm{(E7)}&t_{12}+t_4+t_W+t_{4W}\ge0,&\mathrm{(E8)}&t_{12}+t_3+t_4+t_{34}\ge0.
    \end{array}
\end{equation}
Indeed:
\begin{itemize}
    \item for (E1), sign-average the frames $\{e_1,e_2,e_p,e_q\}$, $p\ne q$ in $W$, then sum over
    the six unordered pairs; each $p$ occurs three times, and the sum is $24t_W$;
    \item for (E2), split $W$ into two pairs in the three ways $\{56\mid78\}$, $\{57\mid68\}$,
    $\{58\mid67\}$; the sign-averaged weakly PIC inequality for each of the three resulting four-frames
    is the sum of its four crossing sectional curvatures, and each of the six pairs in $W$ occurs
    twice, giving $12t_{WW}$;
    \item for (E3), sign-average $\{e_1,e_2,e_3,e_p\}$ and average over $p\in W$; this gives
    $2(t_3+t_W)$;
    \item for (E4), sign-average $\{e_i,e_4,e_p,e_q\}$ for $i\in\{1,2\}$ and $5\le p<q\le8$, then
    sum; the result is $24(t_W+t_{4W})$;
    \item for (E5), sign-average $\{e_1,e_p,e_2,e_q\}$ for ordered $p\ne q$ in $W$ and average over
    the twelve choices; this is $t_{12}+2t_W+t_{WW}$;
    \item for (E6), sign-average $\{e_i,e_3,e_{3-i},e_p\}$ and average over $i=1,2$ and $p\in W$;
    this is $t_{12}+t_3+t_W+t_{3W}$, and the same argument with $e_4$ in place of $e_3$ gives (E7);
    \item for (E8), use the four frames $\{e_i,e_3,e_{3-i},e_4\}$ and $\{e_i,e_4,e_{3-i},e_3\}$,
    $i=1,2$, and sign-average in $e_4$, respectively $e_3$; each of $T_{1313},T_{2323},T_{1414},
    T_{2424}$ occurs twice, while the mixed terms disappear, and their average is (E8).
\end{itemize}

\medskip\noindent\textbf{The linear certificate.}
Define
\begin{equation}\label{eq:gl-K8}
    K:=\frac14\bigl(\beta-\gamma^*\bigr)-\frac12\tau+N_0.
\end{equation}
Since $c_1=\frac14+d$ and $U=2t_{12}+\beta+\gamma^*+4\tau$, this simplifies to
\begin{equation*}
    K=\frac12\bigl(t_{12}+\beta+\tau\bigr)+dU+c_2\Ric_{33}+c_3\Scal.
\end{equation*}
Thus $2K$ has, in the order $(t_{12},t_3,t_4,t_W,t_{34},t_{3W},t_{4W},t_{WW})$, the coefficients
\begin{equation}\label{eq:gl-table8}
    \bigl(1+4d+4c_3,\; 2+4d+4c_2+8c_3,\; 4d+8c_3,\; 2+16d+32c_3,\;
    2c_2+4c_3,\; 8c_2+16c_3,\; 16c_3,\; 24c_3\bigr).
\end{equation}
The decisive linear certificate is
\begin{equation}\label{eq:gl-cert8}
    2K=\mu_8\mathrm{(E8)}+\mu_6\mathrm{(E6)}+\mu_7\mathrm{(E7)}+\mu_5\mathrm{(E5)}
    +\mu_3\mathrm{(E3)}+\mu_4\mathrm{(E4)}+\mu_1\mathrm{(E1)}+\mu_2\mathrm{(E2)},
\end{equation}
where
\begin{equation}\label{eq:gl-mu8}
    \begin{aligned}
        \mu_8&=2c_2+4c_3, & \mu_6&=8c_2+16c_3,\\
        \mu_7&=4d+4c_3-2c_2, & \mu_5&=1-8c_2-20c_3,\\
        \mu_3&=2+4d-6c_2-12c_3, & \mu_4&=2c_2+12c_3-4d,\\
        \mu_1&=12d+14c_2+52c_3-2, & \mu_2&=8c_2+44c_3-1.
    \end{aligned}
\end{equation}
The identity \eqref{eq:gl-cert8} is verified coefficient-by-coefficient: the $t_{34}$ and
$t_{3W}$ coefficients first determine $\mu_8$ and $\mu_6$. Then $t_4,t_{12},t_3,t_{4W},t_{WW},
t_W$ determine $\mu_7,\mu_5,\mu_3,\mu_4,\mu_2,\mu_1$, respectively. In particular, the last
coefficient is
\begin{equation*}
    \begin{split}
        \mu_1={}&(2+16d+32c_3)-(8c_2+16c_3)-(4d+4c_3-2c_2)\\
        &\quad-2\bigl(1-8c_2-20c_3\bigr)-\bigl(2+4d-6c_2-12c_3\bigr)-\bigl(2c_2+12c_3-4d\bigr)\\
        ={}&12d+14c_2+52c_3-2.
    \end{split}
\end{equation*}
All multipliers are strictly positive: $\mu_8,\mu_6>0$ immediately. The first row of
\eqref{eq:gl-ineq8} gives $\mu_7,\mu_5>0$, the second row gives $\mu_4,\mu_2>0$, and the third
row gives $\mu_1,\mu_3>0$. Hence $K\ge0$.

\medskip\noindent\textbf{The Bernstein argument.}
Replacing $e_4$ by $-e_4$ changes $x$ to $-x$ and leaves all other quantities in
the argument unchanged, so it suffices to consider
\begin{equation}\label{eq:gl-x8}
    x\ge0.
\end{equation}
Averaging \eqref{eq:gl-tilt} over $p=5,\dots,8$ gives
\begin{equation}\label{eq:gl-tilt8}
    \Phi(\lambda)+(1-u)\tau\ge0.
\end{equation}
The weakly PIC inequality for the un-tilted frame and \eqref{eq:gl-trace}, respectively, give
\begin{equation}\label{eq:gl-wpic8}
    \beta+\gamma^*-2x\ge0,\qquad N_1\ge0.
\end{equation}
Split off one half of \eqref{eq:gl-tilt8}:
\begin{equation}\label{eq:gl-split8}
    L_\lambda(T)=\frac12\bigl[\Phi(\lambda)+(1-u)\tau\bigr]+Q(\lambda),\qquad
    Q(\lambda):=\frac12\Phi(\lambda)+(1-u)\Bigl(N_0-\frac\tau2\Bigr)+uN_1.
\end{equation}
It remains to prove $Q\ge0$. It is quadratic in $\lambda$:
\begin{equation*}
    Q(\lambda)=Q(0)-x\lambda+\Bigl(\frac{\gamma^*}2+N_1-N_0+\frac\tau2\Bigr)\lambda^2.
\end{equation*}
The Bernstein representation of a quadratic is
\begin{equation}\label{eq:gl-bernstein8}
    Q(\lambda)=(1-\lambda)^2Q(0)+2\lambda(1-\lambda)\Bigl(Q(0)-\frac x2\Bigr)+\lambda^2Q(1),
\end{equation}
and all three Bernstein weights are nonnegative on $[0,1]$. The three control values are
nonnegative:
\begin{equation*}
    Q(1)=\frac12\bigl(\beta+\gamma^*-2x\bigr)+N_1\ge0
\end{equation*}
by \eqref{eq:gl-wpic8}. By $x\le\frac12(\beta+\gamma^*)$ from \eqref{eq:gl-wpic8},
\begin{equation*}
    Q(0)-\frac x2=\frac\beta2-\frac x2+N_0-\frac\tau2
    \ge\frac\beta2-\frac{\beta+\gamma^*}4+N_0-\frac\tau2=K\ge0;
\end{equation*}
and $Q(0)=\bigl(Q(0)-\frac x2\bigr)+\frac x2\ge0$ by \eqref{eq:gl-x8}. Hence $Q(\lambda)\ge0$,
and then \eqref{eq:gl-tilt8} gives $L_\lambda(T)\ge0$. This proves Claim \ref{claim:glstar} for
$n=8$, and by the reduction of Subsection \ref{subsec:proof3-reduction},
\begin{equation*}
    \tilde C\Bigl(\frac1{40}\Bigr)=C\Bigl(\frac1{18}\Bigr)\qquad\text{in }\CB(\mathbb{R}^8).
\end{equation*}

This proves Theorem \ref{thm:gluing}.

\bibliographystyle{plain}

\bibliography{refs}

\end{document}